\documentclass[a4paper, 11pt]{article}
\usepackage{anysize}
\marginsize{3,5cm}{2,5cm}{2,5cm}{2,5cm}
\usepackage{amssymb}
\usepackage{amsmath,amsbsy,amssymb,amscd}
\usepackage{t1enc}
\usepackage[cp1250]{inputenc}
\usepackage[british]{babel}
\usepackage[all]{xy}
\usepackage{color}
\usepackage{amsfonts}
\usepackage{latexsym}
\usepackage{amsthm}
\usepackage{mathrsfs}
\usepackage{hyperref}
\usepackage{graphicx}
\usepackage{comment}

\usepackage{color}
\usepackage{caption}
\usepackage{subcaption}
\usepackage{enumerate}
\usepackage{xcolor}
\usepackage[draft]{todonotes}

\definecolor{orange}{rgb}{1,0.5,0}

\usepackage{mathrsfs} 

\DeclareMathAlphabet{\mathpzc}{OT1}{pzc}{L}{it} 

\theoremstyle{definition}
\newtheorem{definition}{Definition}[section]
\newtheorem{theorem}[definition]{Theorem}

\newtheorem{proposition}[definition]{Proposition}
\newtheorem{corollary}[definition]{Corollary}
\newtheorem{lemma}[definition]{Lemma}
\newtheorem{example}[definition]{Example}

\newtheorem{remark}[definition]{Remark}

\def\C{\mathbb{C}}
\def\cP{\mathcal{P}}
\def\geq{\geqslant}
\def\leq{\leqslant}
\def\R{\mathbb{R}}

\def\Z{\mathbb{Z}}
\def\N{\mathbb{N}}

\def\cB{\mathcal{B}}

\def\epsilon{\varepsilon}

\def\mf{\mathfrak}

\newcommand{\bea}{\begin{eqnarray}}
  \newcommand{\eea}{\end{eqnarray}}
  \newcommand{\beab}{\begin{eqnarray*}}
  \newcommand{\eeab}{\end{eqnarray*}}

  \newcommand{\be}{\begin{equation}}
  \newcommand{\ee}{\end{equation}}

\title{Bernoulli property for homogeneous systems}
\author{Adam Kanigowski}
\date{}

\begin{document}
\maketitle

\begin{abstract}
Let $G$ be a semisimple Lie group with Haar measure $\mu$ and let $\Gamma$ be an irreducible lattice in $G$. {}{For $g\in G$, we consider left translation $L_g$ acting on $(G\slash \Gamma,\mu)$. We show that if $L_g$ is $K$ (which is equivalent to positive entropy of $L_g$) then $L_g$ is a Bernoulli automorphism. As a corollary, we also obtain analogous results for  homogeneous flows}.
\end{abstract}
\tableofcontents
\section{Introduction}
In the theory of dynamical systems, the most random systems are so called Bernoulli shifts, coming from {\em symbolic dynamics}, and determined by independent stationary processes. Smooth systems, by their nature, are not of the above symbolic form. One of the main discoveries in the theory of smooth dynamical systems in the second half of the twentieth century is their possible randomness whose strongest manifestation is expressed by a measure-theoretic isomorphism to a system determined by a stationary independent process. Such systems are called Bernoulli. Another important notion which ``measures'' chaoticity of a system is the $K$ property, introduced by Kolmogorov. By \cite{RoSi}, $K$ property is  equivalent to {\em completely positive entropy}, i.e.\ every non-trivial factor of the system has {\em positive entropy}. Since factors of Bernoulli shifts are Bernoulli, \cite{Orn2}, it follows that Bernoulli systems enjoy the $K$ property. Kolmogorov conjectured that the converse should also be true. This however was disproved by Ornstein in \cite{Orn3}. Moreover, the work of Katok, \cite{Kat6}, provided counterexamples in the smooth category. On the other hand, for many classes of smooth dynamical systems $K$ property implies Bernoullicity. Indeed, Katznelson, \cite{Katznelson}, proved that ergodic toral automorphisms are Bernoulli and this result was then extended by Lind, \cite{Lind}, to  infinite dimensional toral automorphisms. Thomas and Miles, \cite{TM1}, \cite{TM2} and independently Aoki, \cite{AN}, showed that ergodic automorphisms of compact groups are Bernoulli. Bowen, \cite{Bowen}, established the Bernoulli property for Axiom $A$ diffeomorphisms.  Ledrappier, \cite{Ledrappier}, showed that quadratic maps with absolutely continuous invariant measure enjoy the Bernoulli property.

The above principle that ``natural'' smooth systems which are $K$ are in fact Bernoulli was confirmed in the class of flows. Ornstein and Weiss, \cite{OrnsteinWeiss}, proved that geodesic flows on surfaces of constant negative curvature are Bernoulli. This result was extended by Pesin, \cite{Pesin}  to higher dimensional manifolds (without focal points). Ratner, \cite{Ratner111}, established the Bernoulli property for $C^2$ Anosov flows. Then Ratner, \cite{Ratner101}, and Bunimovich, \cite{Bunimovich}, showed that weakly mixing suspensions over Anosov automorphisms are Bernoulli. For  non-uniformly hyperbolic maps and flows (with singularities) the Bernoulli property was proved by Chernov and Haskell, \cite{Chernov}. Moreover, Katok  in \cite{Kat1} constructed smooth Bernoulli systems on any smooth surface; this was later extended to all manifolds in \cite{BFK}.

It is not hard to notice that a fundamental class of smooth systems, namely, the class of {\em homogeneous systems} on quotients of semisimple Lie groups is not listed above. Indeed, the equivalence of $K$ and Bernoulli properties has not been known in it.  S.\ G.\ Dani conjectured in the mid 1970's that the $K$ and Bernoulli properties are equivalent for homogeneous systems (see also (1) in \cite{Orn10}). In \cite{Dani2}, \cite{Dani3}, {}{Dani proved a special case of this conjecture, namely, assuming additionally that the {\em adjoint operator}  of the system is diagonalizable on the central space (we explain the details in Section \ref{subLe}). Since then, there has been no progress on the general case of Dani's conjecture although it often reappeared in the literature, e.g.\ Conjecture 3.2.\ in \cite{Morris}, Section 2.3.\ b in \cite{KSzach}, Problem 37 in \cite{RRR}, Question 11.11 in \cite{Hasel}}. The main aim of this paper is to prove Dani's conjecture in full generality. More precisely, we have the following result:
\begin{theorem}\label{thm:main1}
Let $G$ be a semisimple Lie group, $\Gamma$ an irreducible lattice in $G$ and let $L_g$ be a left translation by $g\in G$ on  $(G\slash \Gamma,\mu)$ that enjoys the $K$ property. Then $L_g$ is a Bernoulli automorphism.
\end{theorem} 

Obviously, a necessary condition for homogeneous systems on irreducible quotients to be $K$ is {\em positive entropy}. In \cite{Dani1}, Dani showed that this necessary condition is also sufficient for the $K$ property. {}{Therefore, Theorem \ref{thm:main1} has the following consequence on the Bernoulli property for homogeneous flows.
\begin{corollary}\label{cor:main1} Let $G$ be a semisimple Lie group, $\Gamma$ an irreducible lattice in $G$ and let $(\Psi_t)$ be a homogeneous flow on $(G\slash \Gamma,\mu)$ with positive entropy. Then $(\Psi_t)$ is a Bernoulli flow. 
\end{corollary}}

{}{
Theorem \ref{thm:main1} and Corollary \ref{cor:main1} by Ornstein's theory \cite{Orn} have the following consequence:
\begin{corollary}\label{cor:2}Entropy is a full invariant of isomorphism in the class of positive entropy homogeneous systems on irreducible quotients of semisimple Lie groups.
 \end{corollary}}
 
{}{We also have the following corollary for general homogeneous systems (without the irreducibility assumption on the lattice):
\begin{corollary} Let $L_g$ be a weakly mixing translation on $(G\slash \Gamma,\mu)$. Then either $L_g$ is Bernoulli or it has a zero entropy homogeneous factor.
\end{corollary}}
 Indeed, notice that $(L_g, G\slash \Gamma,\mu)$ is a finite extension of  the system\\
  $(L_{g_1}\times\ldots \times L_{g_k},\prod_{i=1}^kG_i\slash \Gamma_i, \prod_{i=1}^k\mu_i)$, where $G_i$ are semisimple Lie groups, $\Gamma_i$ are irreducible in $G_i$ and $\mu_i$ is the Haar measure on $G_i\slash \Gamma_i$. By Rudolph's result, \cite{Rud}, it follows that if $L_g$ is weakly mixing then $L_g$ is Bernoulli if and only if all $L_{g_i}$ are Bernoulli. This by Corollary \ref{cor:2} is equivalent to positive entropy of every $L_{g_i}$.

{}{Finally it follows from our result that a strong form of {\em Pinsker's conjecture}\footnote{Pinsker's conjecture states that every automorphism is isomorphic to a Cartesian product of a $K$-system and a system of entropy $0$.} is true for homogeneous systems (although it is not true in general, \cite{Orn0000}):
\begin{corollary} Let $L_g$ be a translation on $(G\slash \Gamma,\mu)$. Then $L_g$ is isomorphic to a product of a Bernoulli system and a system of zero entropy. 
\end{corollary}}
Indeed, by \cite{Field} it follows that finite extension of a system which is a product of Bernoulli and zero entropy is also of that form. It remains to notice that by Corollary \ref{cor:2}, for every $i\in \{1,\ldots k\}$, every $L_{g_i}$ is either Bernoulli or of zero entropy.

 Let us give the simplest example for which Theorem \ref{thm:main1} applies, while \cite{Dani2}, \cite{Dani3} do not.
\begin{example}Let
$$
X=\begin{pmatrix}1&0\\0&-1\end{pmatrix}\;\;\text{ and  }\;\;U=\begin{pmatrix}0&1\\0&0\end{pmatrix}
$$
be the generators of, respectively the {\em geodesic} and the {\em horocycle} flow. Let $L(x,y)=[\exp(X)\times \exp(U)](x,y)\Gamma$, where $\Gamma$ is an irreducible lattice $\in SL(2,\R)^2$.
Then $L$ on $SL(2,\R)^2\slash \Gamma$ is a $K$ system by \cite{Dani1}. However, it is not diagonalizable on $\mf g^0=sl(2,\R)$, so we cannot apply \cite{Dani2}, \cite{Dani3}.  On the other hand Theorem \ref{thm:main1} shows that $L$ is a Bernoulli automorphism. One can also consider the flow $\Psi_t(x,y)=[\exp(tX)\times \exp(tU)](x,y)\Gamma$ and deduce the same result.
\end{example}

Theorem \ref{thm:main1} yields examples of partially hyperbolic, Bernoulli systems for which the center direction is not an isometry (in fact, it may have polynomial growth). The only known such examples until know were ergodic toral automorphisms. Let us briefly describe the novelty of our approach compared to the ones used by \cite{OrnsteinWeiss}, \cite{Ratner101} and \cite{Dani2}, \cite{Dani3} (in all these results, it was crucially used that the center direction is an isometry). One of the main tools in proving the Bernoulli property is the {\em geometric method} developed in \cite{OrnsteinWeiss}. The main difficulty in this method is to find a {\em good matching} between nearby pieces of local unstable manifolds. If the considered system is an isometry on the center space (recall that it contracts the stable space), then a {\em good matching} $\theta$ is given by the {\em central stable holonomy}. It is crucial that the orbits of points $x$ and $\theta x$ stay together for all times. However, if the center is not an isometry, then the center-stable holonomy is not a good matching anymore, since points $x$ and $\theta x$ have different center components and therefore will split after some time. Our method is therefore different from the previously used methods. We use strong equidistribution properties of the unstable foliation (horospherical subgroup) and exact bounds on the growth on the center, to create a matching between local unstable manifolds. Finally, we emphasize that our proof is geometric (it uses effective equidistribution of the unstable foliation and controlled growth on the center). Therefore it has the potential of being generalized to other partially hyperbolic systems with non-isometric center.

\textbf{Outline of the paper:} In Section \ref{sec:bas} we introduce some basic notation and definitions. In Section \ref{sec:pr} we state some preliminary results on homogeneous spaces. In Section \ref{sec:eq} we state a Lemma on equidistribution of the unstable foliation. This Lemma is proven in the appendix. In Section \ref{sec:vwb} we use results from Sections \ref{sec:pr} to show the very weak Bernoulli property. One of the main results is Proposition \ref{mainlemma} in which we construct a map between local unstable manifolds. We provide a general outline of the proof of Theorem \ref{thm:main1} at the end of Section \ref{sec:bas} and an outline of the proof of Proposition \ref{mainlemma} at the beginning of Section \ref{secmainlem}.

\section{Basic definitions}\label{sec:bas}
We will be always dealing with measure preserving actions of $\R$ (flows) and $\Z$ (automorphisms) on standard probability Borel spaces.
\subsection{Bernoulli and very weak Bernoulli properties.}
Let $\mathcal{A}$ be a finite set and let ${\bf p}=(p_i)_{i=1}^{|\mathcal{A}|}$ be a probability vector, i.e.\ $p_i\geq 0$ for $i\in\{1,\ldots, |\mathcal{A}|\}$ and  $\sum_{i=1}^{|\mathcal{A}|}p_i=1$. A Bernoulli shift is a transformation $\sigma$ on $\mathcal{A}^\Z$ given by $\sigma((x_i)_{i\in \Z})=(x_{i+1})_{i\in\Z}$. Notice that $\sigma$ preserves the measure ${\bf p}^\Z$.

An automorphism $T:(X,\cB,\mu)\to (X,\cB,\mu)$ is called a {\em Bernoulli system} (or simply Bernoulli) if it is isomorphic to some Bernoulli shift (i.e.\ for some $\mathcal{A}$ and some probability vector ${\bf p}$.) We say that a flow $(T_t)$ on $(X,\cB,\mu)$ is a {\em Bernoulli flow} (or simply Bernoulli) if for every $t_0\in \R\setminus \{0\}$ the automorphism $T_{t_0}$ is Bernoulli.  We will use the notion of  {\em very weak Bernoullicity} which we now define. First we need to recall some basic notation. Let $\cP=(P_1,\ldots, P_k)$ and $\mathcal{Q}=(Q_1,\ldots,Q_l)$ be two finite partitions of $(X,\mu)$.
Then $\cP\vee \mathcal{Q}$ is the smallest common refinement of $\cP$ and $\mathcal{Q}$, i.e.\ the partition into sets of the form $P_i\cap Q_j$. For $A\subset X$, $\cP_{|A}$ denotes the induced partition of  the space $(A, \mu_{|A})$, i.e.\ $\cP_{|A}:=(P_1\cap A,\ldots, P_k\cap A)$ (here $\mu_{|A}(B)=\frac{\mu(A\cap B)}{\mu(A)}$). Moreover if $k=l$,  we can introduce the following distance on the space of partitions of $(X,\mu)$:
$$
\bar{d}(\cP,\mathcal{Q}):=\sum_{i=1}^{k}\mu(P_i\triangle Q_i).
$$
Now let $\cP^{s}=(\cP^{s}_1,\ldots,\cP^{s}_k)$, $s=1,\ldots, S$ be a sequence of finite partitions of $(X,\mu)$ and $\mathcal{Q}^s=(\mathcal{Q}^s_1,\ldots,\mathcal{Q}^s_k)$, $s=1,\ldots, S$ be a sequence of finite partitions of $(Y,\nu)$.  If  additionally $(X,\mu)=(Y,\nu)$, then
$$
\bar{d}\left((\cP^s)_{s=1}^S,(\mathcal{Q}^s)_{s=1}^S\right):=\frac{1}{S}\sum_{s=1}^S\bar{d}(\cP^s,\mathcal{Q}^s).
$$
More generally, if $(\cP^s)_{s=1}^{S}$ and $(\mathcal{Q}^s)_{s=1}^S$ are partitions of different spaces,  we say that $\cP^s\sim \mathcal{Q}^s$ for $s=1,\ldots, S$ if
$\mu(P_i^s)=\nu(Q^s_i)$ for $i=1,\ldots,k$ and $s=1,\ldots, S$. We can then compare the distance between $(\cP^s)_{s=1}^S$ and $(\mathcal{Q}^s)_{s=1}^S$ by setting
$$
\bar{d}\left((\cP^s)_{s=1}^S,(\mathcal{Q}^s)_{s=1}^S\right)=\inf_{\mathcal{\bar{Q}}^s\sim \mathcal{Q}^s,\; s=1,\ldots,S}\bar{d}\left((\cP^s)_{s=1}^S,(\mathcal{\bar{Q}}^s)_{s=1}^S\right),
$$
where the infimum is taken over sequences of partitions $\mathcal{\bar{Q}}^s$ of $(X,\mu)$.
We say that a property holds for $\epsilon$ a.e.\ atom of a partition $\cP$ if it holds for all atoms except a set of atoms whose union has measure less than $\epsilon$. We denote by $T^n\cP$ the partition given by $(T^nP_1,\ldots, T^nP_k)$.
\begin{definition}[very weak Bernoulli]\label{def:VWB} Let $T\in Aut(X,\cB,\mu)$ and let $\cP$ be a finite partition of $X$. Then $\cP$ is a very weak Bernoulli partition (VWB partition) if for every $\epsilon>0$ there exists $N_0\in \N$ such that  for every $N'\geq N\geq N_0$ every $S\geq 0$ and $\epsilon$ a.e.\ atom of $\bigvee_{N}^{N'}T^{i}(\cP)$, we have
$$
\bar{d}\left(\{T^{-i}\cP\}_{s=0}^S,\{T^{-i}\cP_{|A}\}_{s=0}^S\right)<\epsilon.
$$
\end{definition}
The following classical theorem is a crucial tool in establishing Bernoullicity of a system (see e.g.\ \cite{OrnsteinWeiss}). Recal that a sequence of partitions $(\cP_k)_{k=1}^{+\infty}$ of $(X,\cB,\mu)$ {\em converges to partition into points} if the smallest $\sigma$-algebra with respect to which all $\cP_k$  are measurable, is $\cB$.
\begin{theorem}\label{conga} If $(\cP_k)_{k=1}^{+\infty}$ is a sequence of partitions of $(X,\cB,\mu)$ converging to partition into points and, for every $k\geq 1$, $\cP_k$ is VWB partition for $T\in Aut(X,\cB,\mu)$, then $T$ is a Bernoulli system.
\end{theorem}
We will now recall the main method of establishing VWB property, \cite{OrnsteinWeiss}. For a partition $\cP=(P_1,\ldots P_k)$ of $(X,\mu)$, an integer $S\geq 1$ and $x\in X$ the $S,\cP$-name of $x$ is a sequence $(x^\cP_i)_{i=0}^S\in \{1,\ldots,k\}^{S+1}$  given by the condition $T^{i}(x)\in P_{x^\cP_i}$. Let $e:\Z\to \Z$ be given by e(0)=0 and $e(n)=1$ for $n\neq 0$. A map $\theta:(X,\mu)\to (Y,\nu)$ is called {\em $\epsilon$-measure preserving} if there exists a set $E'\subset  X$, $\mu(E')<\epsilon$ and such that for every $A\in X\setminus E'$, we have
$$
\left|\frac{\nu(\theta(A))}{\mu(A)}-1\right|<\epsilon.
$$

We have the following lemma:
\begin{lemma}[Lemma 1.3. in \cite{OrnsteinWeiss}]\label{lem:VWE} Let $T\in Aut(X,\cB,\mu)$ and $\cP$ be a finite partition of $X$. If for every $\epsilon>0$ there exists $N\in \N$ such that for every $N'\geq N$, $\epsilon$ a.e.\ atom $A\in \bigvee_{N}^{N'}T^i\cP$ and every $S\geq 1$ there exists an $\epsilon$-measure preserving map $\theta=\theta(N,S,A):(A,\mu_{|A})\to (X,\mu)$ such that
$$
d(x,\theta(x)):=\frac{1}{S}\sum_{i=0}^{S-1}e\left(x_i^\cP-(\theta(x))^{\cP}_i\right) <\epsilon.
$$
then $\cP$ is a VWB partition.
\end{lemma}
\bigskip

Finally we recall the definition of {\em Kolmogorov property} ($K$ property for short). 
\begin{definition}\label{Kprop}
We say that $T\in Aut(X,\cB,\mu)$ has the $K$ property if there exists a {\em generating} partition\footnote{This means that $\bigvee_{-\infty }^{+\infty}T^i\cP=\cB$.} $\cP$ such that for every $B\in \cB$ and every $\epsilon>0$ there exists $N_0=N_0(\epsilon,B)$ such that for every $N'\geq N\geq N_0$, $\epsilon$ a.e.\ atom $A\in \bigvee_N^{N'}T^i\cP$ satisfies
$$
|\mu_{|A}(B)-\mu(B)|<\epsilon.
$$
\end{definition}
By \cite{RoSi} it follows that the $K$ property is equivalent to {\em completely positive entropy}: every {\em factor} of $T\in Aut(X,\cB,\mu)$ has positive entropy.

\subsection{Preliminaries on Lie groups and homogeneous spaces}\label{subLe}
Let $G$ be a semisimple Lie group with Haar measure $\mu$, Lie algebra $\mf g$ and denote the right invariant metric on $G$ by $d_G$. Let $\exp:\mf g\to G$ be the exponential map with the local inverse (around $e$) given by $\log_{alg}:G\to \mf g$. Let $[\cdot,\cdot]:\mf g\to \mf g$ denote the Lie bracket. {}{For $g\in G$, let $C(g):G\to G$ be given by $C(g)(h):=ghg^{-1}$ and let  $Ad(g):\mf g\to \mf g$ be the {\em adjoint operator}, i.e.\ derivative of $C(g)$ at $e\in G$. Let $\mf g_\C$ denote the {\em complexification} of $\mf g$ and, for $\lambda\in \C$, let
$$
V_\lambda:=\{U\in \mf g_\C\;:\; \exists_{j\in \N}\; (Ad(g)-\lambda I)^j(U)=0\}
$$
be the {\em generalized eigenspace} of $\lambda$. Then $[V_\lambda,V_\mu]\subset V_{\lambda+\mu}$.
Let
$$
\mf g^+_\C:=\bigoplus_{|\lambda|>1}V_\lambda,\;\;\; \mf g^0_\C:=\bigoplus_{|\lambda|=1}V_\lambda,\;\;\;\mf g^+_\C:=\bigoplus_{|\lambda|<1}V_\lambda.
$$}
and let respectively $\mf g^+,\;\mf g^0,\;\mf g^-\subset \mf g$ be the real spaces corresponding to $\mf g_\C^+,\;\mf g_\C^0,\;\mf g^-_\C$.
We can then decompose $\mf g$:
\be\label{eq:dec}
\mf g:=\mf g^+\oplus\mf g^0\oplus\mf g^-.
\ee
It follows that for $i\in \{+,0,-\}$, $\mf g^i$ is a subalgebra and hence $G^i:=\exp(\mf g^i)$ is a subgroup of $G$. It follows that $G^-$ is the {\em horospherical subgroup} for $L_g$, i.e.\
$$
G^-=\{h\in G\;:\; g^lhg^{-l}\to 0,\text{ as } l\to +\infty\}.
$$
Analogously, $G^+$  is the horospherical subgroup for $L_{-g}$. Moreover, $Ad(g)$ restricted to $\mf g^0$ is called {\em quasi-unipotent}.
Recall also that $\mf g^+\neq \{0\}$ iff $\mf g^-\neq \{0\}$. From now on we fix $g\in G$ for which we assume that $\mf g^+\neq \{0\}$ (equivalently $\mf g^-\neq \{0\}$).

For $i\in\{+,0,-\}$, let ${\bf V}^i=\{V^i_j\}_{j=1}^{\dim \mf g^i}$ be a basis of $\mf g^i$. Then ${\bf V}:={\bf V}^+\cup {\bf V}^0\cup {\bf V}^-$ is a basis of $\mf g$. For $i\in\{+,0,-\}$ and a vector ${\bf a}=(a_j)\in \R^{\dim \mf g^i}$, we denote ${\bf a}{\bf V}:=\sum_{j=1}^{\dim \mf g^i}a_jV^i_j$. Then we define the norm on $\mf g^i$ by settting
$$
\|U\|_i=\max\{|a_j|\;:\;j\in\{1,\ldots,\dim \mf g^i\}\},
$$
where $U={\bf a}{\bf V}$.

For $i\in\{+,0,-\}$, $Leb^i$ denotes the Lebesgue measure on $\mf g^i$. Moreover, $Leb$ denotes the Lebesgue measure on $\mf g$.

\subsection{Homogeneous systems}
Let $\Gamma\subset G$ be a {\em lattice} in $G$.  We define the {\em homogeneous space}, $M:=G\slash \Gamma$.
Then $M$ is equipped with the measure given locally by $\mu$ (we denote the measure on $M$ also by $\mu$). For $g\in G$, the {\em left translation} on $M$  is given by
$$
\Psi(x\Gamma)=L_g(x\Gamma):= (gx)\Gamma.
$$
It follows that $\Psi$ preserves $\mu$. The right invariant metric on $G$ induces the following metric on the homogeneous space:
$$
d_M(x,y):=\inf_{\gamma\in \Gamma}d_G(x,y\gamma).
$$

Our main theorem can be rephrased in the following way:
\begin{theorem}\label{thm:main2} Let $g\in G$ be such that $\mf g^+\neq \{0\}$. If $\Gamma$ is irreducible, then $\Psi$ on $(M,\mu)$ is Bernoulli.
\end{theorem}
By Dani's theorem, \cite{Dani2},\cite{Dani3} it follows that if $\mf g^+\neq \{0\}$ (which is equialent to $h_{\mu}(\Psi)>0$) and ${Ad(g)}_{| \mf g^0}:\mf g^0\to \mf g^0$ is diagonalizable (all generalized eigenspaces are one-dimensional), then $\Psi$ is Bernoulli. Hence our theorem addresses the case in which ${Ad(g)}_{| \mf g^0}:\mf g^0\to \mf g^0$ has non-trivial Jordan blocks.

\subsection{Sobolev norms}\label{sec:sob}
We now briefly recall some basic facts on Sobolev norms. For more details see eg.\ \cite{EMV} or \cite{BEG}.
Let $C_c(M)$ denote the space of compactly supported  functions on $M$ with the topology of uniform convergence on compact sets. Let $C_c^{\infty}(M)\subset C_c(M)$ denote the set of infinitely differentiable functions on $M$.  Recall that for $Y\in \mf g$ and $\phi\in C_c^{\infty}(M)$  the {\em Lie derivative} in direction $Y$ is given by
$$
\mathcal{L}_Y(\phi)(x):=\lim_{t\to 0}\frac{\phi(\exp(tY)x)-\phi(x)}{t}.
$$
Moreover, for a fixed basis $V_1,\ldots, V_n$ of $\mf g$ and a multiindex $(m_1,\ldots,m_n)$, we define the operator of degree deg $W=\sum_{i=1}^nm_i$, by setting
$$
\mathcal{L}_W:=\mathcal{L}^{m_1}{U_1}\ldots \mathcal{L}^{m_n}{U_n}.
$$
Let $r(x)$ denote the {\em injectivity radius} of $x=g\Gamma$, i.e.\
the largest $r>0$ such that the quotient map $\pi:G\to M$ restricted to a closed $d_G$ ball of radius $r$ around $g$ is injective. Then the Sobolev norms of order $d$ are defined by
$$
\mathcal{S}_d(\phi):=\left(\sum_{deg W\leq d}\left\|r(\cdot)^d\mathcal{L}_W(\phi)(\cdot)\right\|^2_{L^2(M)}\right).
$$
Sobolev norms (of order $d$) measure the $L^2$ growth of derivatives of $\phi$ up to level $d$.

We recall the following lemma which establishes quantitative mixing for the action of $G$ on $M$ (recall that we assume that $G$ is semisimple) and is based on the decay of matrix coefficients\footnote{This results are based on work of several authors, Cowling \cite{Cowling}, Howe \cite{Howe}, Moore \cite{Moore}, and Oh \cite{Oh}.}.
\begin{lemma} \label{lem:fff}There exist $C,\zeta,d>0$ such that for any $\phi, \psi \in C_c^\infty(M)$, we have
$$
\left|\int_M \phi(gx)\psi(x)d\mu-\int_M \phi(x)d\mu\int_M \psi(x)d\mu\right|\leq Ce^{-\zeta d_G(g,e)}\mathcal{S}_d(\phi)\mathcal{S}_d(\psi).
$$
\end{lemma}

\subsection{Homogeneous parallelograms, partitions of $M$}
We now define a special class of sets that will be used in the proof of Theorem \ref{thm:main2}. Fix $g\in G$ and recall that for $i\in \{+,0,-\}$,
${\bf V}^i$ denotes the basis of $\mf g^i$. (see \eqref{eq:dec}). For $\mathfrak{h}\in \{\mf g,\mf g^+,\mf g^0,\mf g^-\}$ we define the $\delta$ {\em cube} in $\mathfrak{h}$ by setting
\be\label{eq:cubes}
C(\delta,\mathfrak{h}):=\{V\in \mathfrak{h}\;:\; \|V\|\leq \delta\}.
\ee
 Let $H\subset G$ be the Lie group of $\mathfrak{h}$. Any cube in $\mathfrak{h}$ defines a {\em cube} in $H$ by setting
$$
C(\delta,H):=\exp\left(\{V\in \mathfrak{h}\;:\; \|V\|\leq \delta\}\right).
$$
 For $y\in M$, the $\xi$-{\em local unstable space} is defined by
\be\label{eq:locu}
W^u(\xi,y):=C(\xi,G^+)y \subset M.
\ee
An  {\em HC-set}\footnote{Acronym for homogeneous cube.} of size $\eta>0$ around $x\in M$ is the set
\be\label{eq:hc}
D(\eta,x):=[C(\eta,G^+)\times C(\eta,G^0\times G^-)]x\subset M
\ee
 Let $\eta\in (0,1)$. We say that $E\subset M$ intersects an HC-set  $D(\eta,x)$ in an {\em $\mf g^+$-tubular subset} if $y\in E\cap D$ implies that $W^u(1,y)\cap D\subset  E\cap D$.
\subsection{Partitions of $M$}\label{sec:par}
Fix $\epsilon>0$ and let $K_\epsilon\subset M$ be a subset such that $\mu(K_\epsilon)\geq 1-\epsilon$. Let $r_\epsilon=min(\epsilon,sys(K_\epsilon)/10)$ (see Lemma \ref{lem:systole}). Consider the cover $K_\epsilon\subset \bigcup_{x\in K_\epsilon}B(x,r_\epsilon)$ and let 
$$
K_\epsilon\subset \bigcup_{j=1}^{N(\epsilon)}B(x_j,r_\epsilon)
$$
be a finite subcover. We then make the sets $\{B(x_j,r_\epsilon\}_{j=1}^{N(\epsilon)}$ disjoint by defining $P_j:=B(x_j,r_\epsilon)\setminus \bigcup_{\ell<j} B(x_\ell,r_\epsilon)$. We call the new (disjoint) sets $P_1,...,P_{m(\epsilon)}$, moreover we set $P_0:=M\setminus \left(\bigcup_{j=1}^{N(\epsilon)}B(x_j,r_\epsilon)\right)$. We will be always considering partitions of the form $\{P_{j}\}_{j=0}^{m(\epsilon)}$. Notice that $\partial B(x_j,r_\epsilon)$ is smooth (notice that by the definition on $r_\epsilon$ it follows that the map $\pi: \pi^{-1}(B(x_j,r_\epsilon))\to M$ is injective). Since the sets $\{P_j\}_{j=0}^{m(\epsilon)}$ are obtained from the sets $\{B(x_j,r_\epsilon)\}_{j=1}^{N(\epsilon)}$ by unions, intersections and differences (and there is finitely many of them), it follows that for every $\epsilon'>0$ there exists $\epsilon''>0$ such that 
\be\label{eq:masss}
\mu(V_{\epsilon''}(\partial \cP))\leq \epsilon'/20,
\ee
where $\partial\cP=\bigcup_{j=0}^{m(\epsilon)}\partial P_j$, and $V_{\epsilon''}(\partial \cP)$ denotes the $\epsilon''$ neighborhood if $\partial \cP$.

Notice that if $\epsilon_n=\frac{1}{n}$ then the sequence $\{P_j\}_{j=1}^{m(\frac{1}{n})}$ for $n\in \N$ converges to partition into points and hence it is enough to establish the VWB property along $\{P_j\}_{j=1}^{m(\frac{1}{n})}$, $n\in \N$. From now on we always assume that we have fixed an $\epsilon_0=\frac{1}{n_0}$ and $\cP$ always denotes the partition $\{P\}_{j=1}^{m(\frac{1}{n_0})}$.
\\
To make the paper more readable, we will now present a general strategy of the proof of Theorem \ref{thm:main2}.
\subsection{Outline of the proof of Theorem \ref{thm:main2}:} 
We outline some standard reductions introduced in \cite{OrnsteinWeiss} in the problem of VWB-property for $K$- automorphisms.  By Lemma \ref{lem:VWE} one needs to show existence of an $\epsilon$-- measure preserving map $\theta:(A,\mu_{|A})\to (X,\mu)$ for $\epsilon$ a.e.\ atom $A\in \bigvee_{N}^{N'}\Psi^i(\cP)$ such that the orbits of $x$ and $\theta(x)$ are in one atom of $\cP$ for most times. The first step is to use the $K$ property of $\Psi$ (see Definition \ref{Kprop}) to reduce the general problem to a local one, i.e.\ it is enough to find an $\epsilon$-measure preserving map $\bar{\theta}_D:(A\cap D,\mu_{|A\cap D})\to (D,\mu_{|D})$, where $D$ belongs to a family of  small {\em  cubes}  in $M$ (smallness depending on $\epsilon$) whose union has large measure and then use the $K$ property to extend the family of obtained $\epsilon$ preserving maps to a global map $\theta$ (see Proposition \ref{prop:match}). To construct the local map $\bar{\theta}_D$ one uses Lemma \ref{lem:intersect} to say that $\epsilon$ a.e.\ atom $A$ intersects $D$ in an $\mf g^+$ tubular subset. This by the local product structure (see the proof of Proposition \ref{prop:match}) reduces the local problem to finding an $\epsilon$-measure preserving map $\tilde{\theta}:W^u(\delta,z)\to W^u(\delta,z')$ between pieces of two local unstable manifolds in $D$ (with conditional measures $\mu^u_z$ and $\mu^u_{z'}$). 

Finding the map $\tilde{\theta}$ is the most difficult and important part in proving VWB-property (it is called the {\em Main Lemma} in \cite{OrnsteinWeiss}). This is done in Proposition \ref{mainlemma} (for most unstable pieces, for which we have good control of ergodic properties). This is the part on which most of the proof is devoted. The main obstruction in finding $\tilde{\theta}$ is non-trivial (polynomial) orbit growth on the center space (otherwise, if the center is isometry, one can take for $\tilde{\theta}$ the {\em center stable holonomy}). The idea is to use quantitative equidistribution of the local unstable space. At time $N$ we divide the space $M$ into a family of cubes of diameter $\frac{1}{N^r}$ (where $r$ is the exponent of the growth on the center space). If we can guarantee that  unstable pieces (of most points) equidistribute at time $N$ for all the cubes of order $\frac{1}{N^r}$ one can construct the mapping $\tilde{\theta}$ by mapping subsets of $W^u(\delta,z)$ that belong to the $i$-th cube at time $\epsilon^3N$ to subsets of $W^u(\delta,z')$ that belong to the $i$-th cube (using the holonomy {\em inside} each cube). Quantitative equidistribution estimates imply that such map is $\epsilon$-measure preserving (for a more detailed outline of the construction of $\tilde{\theta}$ see the outline of the proof of Proposition \ref{mainlemma} in Section \ref{secmainlem}). The above strategy shows, that one needs quantitative equidistribution for families of cubes. We state such result in Section \ref{sec:eq} (see Lemma \ref{eq:distcube}). The proof of Lemma \ref{eq:distcube}, based on Proposition 2.4.8. from \cite{KlMa}, is given in the appendix. Then in Section \ref{sec:vwb}, we use the above outline to reduce the problem to finding $\epsilon$-measure preserving map between pieces of unstable manifolds (see Proposition \ref{mainlemma}).

\section{Preliminary results on homogeneous spaces}\label{sec:pr}
In this section we will recall some basic results from Lie groups, Lie algebras and homogeneous spaces.
\subsection{Basic lemmas in Lie theory}
We first state some classic results on norms of vectors in Lie algebra. In the lemma below  $G$ is a Lie group with Haar measure $\mu$ and $\mf g$ is the Lie algebra of $G$. Let $Leb$ denote the Lebesgue measure on $\mf g$.
\begin{lemma}\label{lema:Jac}The following hold:
\begin{enumerate}
\item[${(i)}$] Let $\mf g=\mathfrak{e}_1\oplus...\oplus \mathfrak{e}_k$ be any decomposition of $\mf g$ into a sum of vector subspaces. If for $g\in G$, $d_G(g,e)$ is sufficiently small, then $g=\exp(E_1)\cdot\ldots\cdot\exp(E_k)$, with $E_i\in \mathfrak{e}_i$ for every $i\in \{1,\ldots,k\}$. Moreover such decomposition is unique.
\item[$(ii)$] there exist $\kappa_1,\kappa_2>0$ such that if $\|(U_1,\ldots,U_n)\|<\kappa_2$, then
$$
\left\|\log_{alg}\left(\exp(U_1)\ldots\exp(U_n)\right)-\sum_{i=1}^n U_i\right\|\leq \kappa_1 \max_{i,j\in\{1,\ldots, n\}}\|U_i\|\|U_j\|.
$$
\item[$(iii)$] let $B\subset \mf g$. For every $\epsilon'>0$ there exists $\epsilon''>0$ such that if $\sup_{V\in B}\|V\|<\epsilon''$ then
$$
(1-\epsilon')Leb(B)\leq \mu(\exp(B))\leq (1+\epsilon')Leb(B).
$$
\item[$(iv)$] for every $\epsilon'>0$ there exists $\epsilon''>0$ such that for every $0<\tau<\epsilon''$, we have
$$
C((1-\epsilon')\tau,\mf g)\subset C(\tau,\mf g^+)\times C(\tau,\mf g^0\oplus g^-)\subset C((1+\epsilon')\tau,\mf g).
$$
\end{enumerate}
\end{lemma}
\begin{proof} Properties $(i)$ and $(iii)$ are a consequence of the fact that $exp:\mf g\to G$ is a smooth function which satisfies $\exp'(0)=Id$, for a more detailed argument see e.g.\ \cite{KanigowskiVinhageWei2}, Lemma 3.7. Property $(ii)$ is a consequence of the Taylor formula. Finally $(iv)$ is a straightforward consequence of $(ii)$.
\end{proof}

\begin{remark} We will be constructing $\epsilon$- measure preserving maps between subsets of $M$ (with small diameter). Property $(iii)$ in the above lemma gives an important tool in doing this (on the local level): it is enough to construct an $\epsilon$-measure preserving map on the Lie algebra level (which is a vector subspace), and then send it by $\exp$ to  $M$, where the above lemma implies that the new map will be $2\epsilon$-measure preserving.
\end{remark}

We now define the holonomy maps, which play an important role in the proof of theorem. Let $hol_s:\mf g^+\oplus[\mf g^0\oplus \mf g^-]\to \mf g^+$, $hol_{uc}: \mf g^+\oplus[\mf g^0\oplus \mf g^-]\to\mf g^0\oplus \mf g^-$ be given by
\be\label{eq:holonomymap}
\exp(E)\exp(F)=\exp(hol_{uc}(E,F))\exp(hol_s(E,F)).
\ee

The following result is an immediate consequence of $(ii)$ in Lemma \ref{lema:Jac}.
\begin{lemma}\label{lem:hol}There exist $\kappa_1,\kappa_2>0$ such that for $(E,F)\in \mf g^+\oplus [\mf g^0\oplus \mf g^-]$, with $\|(E,F)\|\leq \kappa_1$ the functions $hol_s(E)$ and $hol_{uc}(F)$ are well defined. Moreover  
$$
\max\left(\|hol_s(E,F)-E\|,\|hol_{uc}(E,F)-F\|\right)\leq \kappa_2\|E\|\|F\|.
$$
\end{lemma}
\begin{proof} It is enough to use $(ii)$ in Lemma \ref{lema:Jac} for $(U_1,U_2)=(E,F)$ and then $(U_1,U_2)=(hol_s(E,F),hol_{uc}(E,F))$.
\end{proof}

\begin{remark}\label{rem:hol} The function $hol_s:\mf g^+ \to \mf g^+$ 
is called the {\em stable holonomy} and plays an important role in proving VWB Bernoulli property. The above lemma implies that if we consider the function $h_s$ on a small cube, then the image, up to a subset of small measure, remains inside the cube.
\end{remark}

For a linear operator $\Psi$ on a vector space $(V,\|\cdot\|)$, we denote
$||\Psi||=\sup_{v\in V\setminus 0}\frac{\|\Psi(v)\|}{\|v\|}$.

The following lemma describes the behavior of the action of $Ad(g):\mf g\to \mf g$.

\begin{lemma}\label{lem:behad} There exist $c_0>0$, $\lambda_0>1$, $r'>0$ such that
\begin{itemize}
\item[{\bf A.}] for every $t<0$, $||Ad(g)_{| \mf g^+}||\leq c_0\lambda_0^{t}$;
\item[{\bf B.}] for every $t>0$, $||Ad(g)_{| \mf g^-}||\leq c_0\lambda_0^{-t}$;
\item[{\bf C.}] for every $t\in \R$,
 $
||Ad(g)_{| \mf g^0}||\leq  c_0|t|^{r'}.
$
\end{itemize}
\end{lemma}
\begin{proof} Notice that \textbf{A.} follows from the fact that for $Ad(g)$, we have $|\lambda|>1$ on $\mf g^+$ and analogously \textbf{B.} follows since $|\lambda|<1$ on $\mf g^-$. So we only need to show \textbf{C.} This however follows from the fact that the growth of the $Ad(g)$ on $\mf g_0$ is at most polynomial, since $|\lambda|=1$ on $\mf g_0$.
\end{proof}
We have the following immediate Lemma:

\begin{lemma}\label{lem:quasi}
There exists $r_0\in \N$ such that for every $\epsilon>0$ there exists  $N_\epsilon>0$ such that for every $x\in M$ and every $y\in C(N^{-r_0},G^0\times G^-)x\in M$ with $N\geq N_\epsilon$, we have
$$
d_M(\Psi^n x,\Psi^n y)<\epsilon \text{ for every } n\in [0,N].
$$
\end{lemma}
\begin{proof}By right invariance, we have
$$
d_G(\Psi^n x,\Psi^n y)=d_G(e,g^nyx^{-1}g^{-n}).
$$
Then the statement follows from \textbf{B.} and \textbf{C.} in Lemma \ref{lem:behad}: we have that the growth of $\|Ad^n(g)\|$ is at most polynomial in $n$ (actually it is bounded on $\mf g^-$ and at most polynomial on $\mf g^0$). Therefore, there exists $r_0\in\N$ (for instance $r_0=r'+1$, where $r'$ comes from \textbf{C.}) and, for every $\epsilon>0$ there exists $N_\epsilon\in \N$ such that if $N\geq N_\epsilon$, then we have
$$
d_G(e,g^nyx^{-1}g^{-n})\leq \frac{C}{N}.
$$
This finishes the proof.
\end{proof}

Finally, we recall a classical result which uses only the fact that $\Gamma$ is a discrete subroup in $G$.
\begin{lemma}\label{lem:systole} For every compact set $K\subset M$, we have
$$
sys(K):=\inf_{x\in K}\inf_{\gamma\in \Gamma\setminus\{e\}}d_G(x\gamma x^{-1},e)>0.
$$
\end{lemma}

 \subsection{Tubular intersections, families of cubes}
   The following lemma is a generalization of Lemma 2.1.\ in
  \cite{OrnsteinWeiss}, it is similar to Lemma in \cite{Dani3}. We will provide a  proof for completeness.
  
\begin{lemma}\label{lem:intersect} Let $\cP$ be a partition of $M$ and let  $D=D(\eta,x)$ be a fixed HC-set (for fixed $\eta>0$ and $x\in M$). For every $\epsilon''>0$ there exists $N_{\epsilon''}\in \N$ such that for every $N'\geq N\geq N_{\epsilon''}$ and for $\epsilon''$ a.e.\ atom $A\in \bigvee_{i=N}^{N'}\Psi^i(\cP)$ there exists $E\subset A$ with $\mu(E)\geq (1-\epsilon'')\mu(A)$ and such that
$E$ intersects $D$ in an $\mf g^+$- tubular subset.
\end{lemma}
\begin{proof}The proof follows the scheme of the proof of Lemma 2.1.\ in \cite{OrnsteinWeiss} and uses (exponential) contraction of the local unstable space for negative times. More precisely, define $G_k$ by setting  $x\in G_k$ if there exists an atom $\Psi^kA\in \Psi^k\cP$ such that $x\in D\cap \Psi^kA$ but $W^u(\xi,x)\cap D$ is not a subset of $D\cap \Psi^kA$. Notice that this is equivalent, to $\Psi^{-k}(D\cap W^u(\xi,x))$ being $d_k$ close to the boundary of $A$, where, by $\textbf{A.}$ in Lemma \ref{lem:behad} and \eqref{eq:masss},
$$
 d_k\leq c_0 \lambda_0^{-k}.
$$
Indeed, for $\exp(B)x\in W^u(\xi,x)$, $B\in C(\xi, G^+)$, we have
$$
d_M(\Psi^{-k}x,\Psi^{-k}\exp(B)x)\leq d_G(\Psi^{-k}\exp(B)\Psi^k,e)
$$
and we use $\textbf{A}$. Using this, the fact that  $\mu$ is  $\Psi$ and also using that $A$ has smooth boundaries it follows that
$$
\mu(G_k)\leq c'_0\lambda^{-k},
$$
for some constant $c'_0>0$. Let $N_{\epsilon''}$ be such that if $G:=\bigcup_{k\geq N_{\epsilon''}}^{+\infty}G_k$, then $\mu(G)<\epsilon''^2$.
Using Markov's inequality, it follows that $\epsilon''$ almost every atom $A\in \bigvee_{N}^{N'}\Psi^k\cP$, we have $\mu(A\cap G)\leq \epsilon'' \mu(G)$(otherwise by summing over this $A$ would contradict $\mu(G)<\epsilon''^2$). This implies that it is enough to define $E:=A\cap G^c$, and then $E$ satisfies the assertion of the lemma.
\end{proof}

The following lemma will be needed in the proof of Theorem \ref{thm:main2}. This lemma is not needed in Dani's paper, \cite{Dani3}, i.e.\ if $Ad(g)$ is semisimple on $\mf g^0$ but plays an important role otherwise, recall \eqref{eq:hc}.

\begin{lemma}\label{lem:ext}For every $\epsilon'>0$ there exists $\epsilon''>0$ such that  for every $0<\delta< \epsilon''$  there exists a finite set $\{x_j\}_{j=1}^{N(\delta)}\subset M$ such that:
\begin{itemize}
\item[${\bf d1.}$] for $i\neq j$, we have $D(\delta,x_j)\cap D(\delta,x_i)=\emptyset$;
\item[${\bf d2.}$] we have $\mu\left(\bigcup_{j=1}^{N(\delta)}D(\delta,x_j)\right)\geq 1-\epsilon'$;
\item[${\bf d3.}$] there exists a compact set $K_{\epsilon'}$ such that for every $\delta\in (0,\epsilon'')$, $\bigcup_{j=1}^{N(\delta)}D(\delta,x_j)\subset K_{\epsilon'}$.
\end{itemize}

\end{lemma}
\begin{proof}By $(iv)$ in Lemma \ref{lema:Jac} it follows that it is enough to show  the existence of a family
$\{C(\delta,G)x_j\}_{j=1}^{N(\delta)}$ satisfying ${\bf d1.}$, ${\bf d2.}$ and ${\bf d3.}$ (for $\epsilon'=\epsilon'/10$). Let $\bar{K}_{\epsilon'}\subset M$ be a compact set with $\mu(\bar{K}_{\epsilon'})>1-\epsilon'/10$. Similarlyy to section \ref{sec:par}, we cover $\bar{K}_{\epsilon'}\subset \bigcup_{x\in \bar{K}_{\epsilon'}}B(x,r_{\epsilon'})$ and let 
$$
\bar{K}_{\epsilon'}\subset \bigcup_{j=1}^{m(\epsilon')}B(x_j,r_{\epsilon'})
$$
be a finite subcover. Let then $K_{\epsilon'}$ be the closure of $\bigcup_{j=1}^{m(\epsilon')}B(x_j,r_{\epsilon'})$, i.e.
$$
K_{\epsilon'}=\overline{\bigcup_{j=1}^{m(\epsilon')}B(x_j,r_{\epsilon'})}.
$$ 
 Notice that $\mu(K_{\epsilon'})\geq \mu(\bar{K}_{\epsilon'})\geq 1-\epsilon'/10$. Analogously to the construction of the sets $\{P_j\}_{j=1}^{m(\epsilon')}$ in Section \ref{sec:par}, let  $\mathcal{R}=\{R_k\}_{k=1}^{W(\epsilon')}$  be any partition of $ K_{\epsilon'}$  such that $\mu(R_1)\leq \epsilon'/10$ and every $R_k$, $k\geq 2$ has (piecewise) smooth boundary and $diam(R_k)<\min\left(sys(K_{\epsilon'})^2,\epsilon'^2\right)$ (see Lemma \ref{lem:systole}). We will WLOG assume that $sys(\epsilon')=sys(K_{\epsilon'})<\epsilon'^2$ (if not, we consider $\tilde{sys}(\epsilon)=\min(sys(\epsilon),\epsilon^2)$. We will show that for every $\delta$ small enough and every $k\in\{2,\ldots W\}$, we can find a family $\{C(\delta,G)x'_j\}_{j=1}^{p(k,\delta)}$ such that $C(\delta,G)x'_j\subset R_k$ for every $j\in \{1,\ldots,p\}$,
$C(\delta,G)x'_j\cap C(\delta,G)x'_{j'}=\emptyset$ for $j\neq j'$ and
\be\label{eq:esrt}
\mu\left(\bigcup_{j=1}^{p}C(\delta,G)x'_j\right)\geq (1-\epsilon')\mu(R_k).
\ee
The general result then follows by taking the union over all $k\in \{2,\ldots,W\}$ of the families $\{C(\delta,G)x'_j\}_{j=1}^{p(k,\delta)}$.

Let $x$ be any point in $R_k$. and let  $Gr:=[2(1+\epsilon')\delta]\Z^{\dim \mf g}\cap[0,\frac{sys(\epsilon')}{3}]^{\dim\mf g}$ be a grid $\in [0,\frac{sys(\epsilon')}{3}]^{\dim \mf g}$ of size $2(1+\epsilon')\delta$. Let $ Gr'\subset Gr$ be  such that for any ${\bf a}\in Gr'$, we have $x_{{\bf a}}:=\exp({\bf a}{\bf V})x \in R_k$ and  also $C(\delta,G)x_{{\bf a}}\subset R_k$. We then define the family
\be\label{def:fam}
\left\{C(\delta,G)x_{{\bf a}}\right\}_{\bf a\in Gr'}.
\ee
Notice first that if ${\bf a}, {\bf a'}\in Gr'$ then
\be\label{eq:emy}
C(\delta,G)x_{{\bf a}}\cap C(\delta,G)x_{{\bf a'}}=\emptyset.
\ee
Indeed, if not then for some $\gamma \in \Gamma$ and $g,g'\in C(\delta,G)$, we have
$$
g\exp({\bf a}\mathcal{V})x=g'\exp({\bf a'}\mathcal{V})x\gamma.
$$
If $\gamma=e$, then the above equation becomes
$$
g'^{-1}g=\exp({\bf a'}\mathcal{V})\exp(-{\bf a}\mathcal{V}).
$$
Since $g',g\in C(\delta, G)$, by $(ii)$ in  Lemma \ref{lema:Jac} it follows that
$$
\|\log_{alg}(g'^{-1}g)\|\leq 2\delta +{\rm O}(\kappa_2 \delta^2)<2(1+\epsilon'/10)\delta.
$$
Similarly, we write ${\bf a'}\mathcal{V}={\bf a}\mathcal{V}+[{\bf a'}-\bf{a}]\mathcal{V}$ and using $[{\bf a}\mathcal{V},{\bf a}\mathcal{V}]=0$, the fact that $\|{\bf a'}-\bf{a}\|=2(1+\epsilon')\delta$ and ${\bf a}\leq sys(\epsilon')/3$, by $(ii)$ in Lemma \ref{lema:Jac}, we get
$$
 \|\log_{alg}(\exp({\bf a'}\mathcal{V})\exp(-{\bf a}\mathcal{V}))\|\geq 2(1+\epsilon)\delta+ {\rm O}(2(1+\epsilon)\kappa_2 sys(\epsilon)/3\delta)>2(1+\epsilon/3)\delta.
$$
The two last inequalities are contradictory and this finishes the proof if $\gamma=e$.
On the other hand, if $\gamma\neq e$, then the above equation transforms to
$$
x\gamma x^{-1}=\exp(-{\bf a'}\mathcal{V})g'^{-1}g\exp({\bf a}\mathcal{V}).
$$
Notice however that by triangle inequality, we have
$$
d_G(\exp(-{\bf a'}\mathcal{V})g'^{-1}g\exp({\bf a}\mathcal{V}),e)\leq 2\delta+2sys(\epsilon')/3,
$$
which contradicts the fact that $x\in K_{\epsilon'}$ (since, by Lemma \ref{lem:systole}, $d_G(x\gamma x^{-1},e)>sys(\epsilon')$) if $\delta$ is small enough. So \eqref{eq:emy} indeed holds. So it remains to show \eqref{eq:esrt} for the family in \eqref{def:fam}.

For this aim, we will first show that
\be\label{lst}
\{y\in R_k\;:\; d_M(y,\partial R_k)>4(1+\epsilon')\delta\}\subset
\bigcup_{{\bf a}\in Gr'}C((1+3\epsilon')\delta,G)x_{{\bf a}}.
\ee
Indeed, notice first that by the definition of $Gr'$ it follows that if $d_M(x_{\bf{a}},\partial R_k)>2\delta$, then $C(\delta,G)x_{{\bf a}}\subset R_k$. Take $y\in R_k$ such that $d_M(y,\partial R_k)>\delta$.
Since $diam(R_k)<sys(\epsilon)^2$ it follows that there exists $\gamma\in \Gamma$ and ${\bf q}\in \left[0,\frac{sys(\epsilon')}{3}\right]^{\dim \mf g}$ such that $y=\exp({\bf q}{\bf V})x\gamma$. Let ${\bf a''}\in Gr$ be a point which minimizes $\|{\bf q}-{\bf a}\|$ over  all ${\bf a}\in Gr$. By the definition of the set $Gr$ it follows that $\|{\bf q}-{\bf a''}\|\leq (1+\epsilon')\delta$. Notice that
$$
y=\exp({\bf q}{\bf V})x\gamma=\exp([{\bf q}-{\bf a}]{\bf V}+{\bf a}{\bf V})x\gamma.
$$
By Lemma $(ii)$ in Lemma \ref{lema:Jac}  and since $[{\bf a''}{\bf V},{\bf a''}{\bf V}]=0$ it follows that
\begin{multline*}
\left\|\log_{alg}\left[\exp([{\bf q}-{\bf a''}]{\bf V}+{\bf a}''{\bf V})\exp(-{\bf a''}{\bf V})\right]\right\|= {\bf q}-{\bf a''}+{\rm O}((1+\epsilon')\delta\|\bf a''\|)\leq\\ (1+\epsilon')\delta+ \epsilon'^2(1+\epsilon)\delta<\delta+2\epsilon'\delta.
\end{multline*}

Therefore  $y\in C((1+3\epsilon')\delta,G)x_{{\bf a''}}$. It remains to show that $\bf{a''}\in Gr'$ (which is equivalent to $C(\delta,G)x_{{\bf a''}}\subset R_k$). Indeed, notice that
\begin{multline*}
d_M(x_{{\bf a''}},\partial R_k)\geq d_M(y,\partial R_k)-d_M(y,\exp({\bf a''}\mathcal{V})x)\geq \\
4(1+\epsilon')\delta-d_G(\exp({\bf q}\mathcal{V}),\exp({\bf a''}\mathcal{V}))\geq 4(1+\epsilon')\delta-2(1+\epsilon')\delta=2(1+\epsilon')\delta.
\end{multline*}
Consequently, $C(\delta,G)x_{{\bf a''}}\subset R_k$ and hence ${\bf a''}\in Gr'$. This finishes the proof of \eqref{lst}. By $(iii)$ in Lemma \ref{lema:Jac}, \eqref{lst} and \eqref{eq:emy}, we have  
\begin{equation}
\begin{aligned}
&\frac{1+\epsilon'}{1-\epsilon'}(1+3\epsilon')^{\dim\mathfrak{g}}\mu\left(\bigcup_{{\bf a}\in Gr'}C(\delta,G)x_{{\bf a}}\right)=\sum_{{\bf a}\in Gr'}\frac{1+\epsilon'}{1-\epsilon'}(1+3\epsilon')^{\dim\mathfrak{g}}\mu\left(C(\delta,G)x_{{\bf a}}\right)\\&\geq\sum_{{\bf a}\in Gr'}\mu\left(C((1+3\epsilon')\delta,G)x_{{\bf a}}\right)\geq\mu\left(\bigcup_{{\bf a}\in Gr'}C((1+3\epsilon')\delta,G)x_{{\bf a}}\right)\geq\\&\mu\left(\{x\in R_k\;:\; d_M(x,\partial R_k)>4(1+\epsilon')\delta\}\right)\geq (1-\phi(\delta))\mu(R_k),
\end{aligned}
\end{equation}
where $\lim_{\delta\to 0}\phi(\delta)=0$ (since the boundary of $R_k$ is piecewise smooth). So taking sufficiently small $\delta$
(if neccesary, we can make  $\delta$ smaller when $k$ ranges over $\{2,\ldots, W\}$) and changing $\epsilon'=\epsilon'''$ so that $1-\epsilon''';\geq \frac{1+\epsilon'}{1-\epsilon'}(1+3\epsilon')^{\dim\mathfrak{g}}$ proves \eqref{eq:esrt} and hence finishes the proof of the lemma.
\end{proof}

\subsection{Equidistribution estimates}\label{sec:eq}
In this section we will establish quantitative (polynomial) equidistribution 
estimates of the horospherical subgroup under the action of $\Psi$.  We will consider the family of cubes from Lemma \ref{lem:ext} i.e.\ for $\delta_k= (\log k)^{-r'}$, let
$$
\left\{D\left(\delta_k,x_j\right)\right\}_{j=1}^{N_k}
$$
be the family from the statement of Lemma \ref{lem:ext}. Notice that by ${\bf d1}$ and ${\bf d2}$ it follows that $N_k\leq (\log k)^{r'\dim \mf g}$. Let morever $K_\epsilon$ be the compact set from ${\bf d3}$.
We have the following lemma:

\begin{lemma}\label{eq:distcube}For every $\epsilon>0$ and every $\xi\in (0,\epsilon)$ there exists $k_\epsilon$ such that for every $k\geq k_\epsilon$, every $N\geq \log k$ and every $x\in K_\epsilon$, we have
\begin{multline}\label{eq:KS}
(1-\epsilon^2)Leb^+(C(\xi,\mf g^+))\mu\left(D\left(\delta_k,x_j\right)\right)\leq \\
Leb^+\left(P\in C(\xi,\mf g^+)\;:\; \psi^N\exp(P)x\in D\left(\delta_k,x_j\right)\right)\leq\\
 (1+\epsilon^2)Leb^+(C(\xi,\mf g^+))\mu\left(D\left(\delta_k,x_j\right)\right).
\end{multline}
\end{lemma}

The result follows from Proposition 2.4.8. in \cite{KlMa} together with an approximation argument. We will give the proof in the appendix.

\section{Proof of the VWB property}\label{sec:vwb}
\subsection{Good returns for most points}
 Recall that the partition $\cP$ is now fixed (see Subsection \ref{sec:par}). We will use Lemma \ref{lem:VWE} to establish the VWB property for $\Psi$. Fix $\epsilon>0$. Let  ${\bf D}=\{D(\epsilon'',x_j\}_{j=1}^{N(\epsilon'')}$ satisfying ${\bf d1}$, ${\bf d2}$ and ${\bf d3}$ (we wlog assume that $\epsilon''\leq \epsilon^{100}$). Notice that since $\epsilon''$ depends on $\epsilon$ only, by the $K$-property of $\Psi$ (see Definition \ref{Kprop}) with $B_j=D(\epsilon'',x_j)$ and $\epsilon=\frac{\epsilon\mu(D_j)}{10N(\epsilon'')}$, we get that there exists $N_j(\epsilon,B_j)$ such that for every $N'\geq N\geq N_j$ and for $\frac{\epsilon\mu(D_j)}{50N(\epsilon'')}$ a.e.\ atom  $A\in \bigvee_{N}^{N'}\Psi^i(\cP)$, we have
\be\label{eq:ada0}
\left|\frac{\mu(A\cap D_j)}{\mu(A)}-\mu(D_j)\right|\leq \frac{1}{50}\epsilon\mu(D_j). 
\ee
Defining $N_2'=\max_{j\in\{1,\ldots, N(\epsilon'')\}} N_j$, by the above it follows that for every $N'\geq N\geq N_2'$ and for $\epsilon$ a.e.\ atom  $A\in \bigvee_{N}^{N'}\Psi^i(\cP)$,we have
\be\label{eq:ada}
\left|\frac{\mu(A\cap D_j)}{\mu(A)}-\mu(D_j)\right|\leq \frac{1}{10}\epsilon\mu(D_j),
\ee
for all $j\in \{1,\ldots,N(\epsilon'')\}$.

Let $\partial \cP=\bigcup_{P\in \cP}\partial P$, and let $V_{\epsilon'}(\partial \cP)$ be the $\epsilon'$ neihborhood of $\partial \cP$,where $\epsilon'$ is such that $\mu(V_{\epsilon'}(\partial \cP))<\epsilon/20$ (see \eqref{eq:masss} in Subsection \ref{sec:par}). Since $\epsilon'$ is an explicit function of $\epsilon$ we will below use $\epsilon$ instead of $\epsilon'$. Let $p=\dim \mf g^+$.
The following proposition is important in the proof of Theorem \ref{thm:main2}:
\begin{proposition}
\label{prop:match}Let $D\in {\bf D}$. There exists $N_3$ such that for every $S,N\geq N_3$ and $\epsilon/50$ a.e.\ atom $A\in  \bigvee_{N_3}^N\Psi^i(\cP)$ there exists an $\epsilon/50$ -measure preserving map $\theta_S:(A\cap D,\mu_{|A\cap D})\to (D,\mu_{|D})$ such that
\be\label{eq:aas}
d(\Psi^k(x),\Psi^k(\theta_Sx))<\epsilon/10\text{ for  } (1-\epsilon^2) \text{ proportion of } k\in \{0,\ldots,S-1\}.
\ee
\end{proposition}

Before we prove Proposition \ref{prop:match} let us show how it implies Theorem \ref{thm:main2}:
\begin{proof}[Proof of Theorem \ref{thm:main2}] We will use Lemma \ref{lem:VWE}.  By ergodic theorem for the characteristic function of $V_{\epsilon'}(\partial \cP)$ there exists a set $\tilde{A}_\Psi$, $\mu(\tilde{A}_\Psi)>1-\epsilon^8$ and $N_4=N_4(\epsilon)$, such that for every $x\in \tilde{A}_\Psi$ and every $N\geq N_4$, we have
\be\label{eq:betk}
{\rm card}\left\{i\in\{0,\ldots, N-1\}\;:\; \Psi^ix\notin V_{\epsilon'}(\partial \cP)\right\}\geq (1-\epsilon/2)N
\ee
  Let $N=\max(N_2',N_3,N_4)$ where $N_2'$ comes from \eqref{eq:ada} and  $N_3$ comes from Proposition \ref{prop:match}).
  
 Let $B_1=\cup_{D\in {\bf D}\setminus{\bf D'}}D$, where ${\bf D'}\subset  {\bf D}$ is such that for $D\in {\bf D'}$, we have
 $\mu(D\cap \tilde{A}_\Psi)\leq (1-\epsilon)\mu(D)$. Since $\mu(\tilde{A}_\Psi)\geq 1-\epsilon^8$, it follows that $\mu(B_1)\geq 1-\epsilon^3$.

 We say that $A\in  \bigvee_{N'_2}^N\Psi^i(\cP)$ is {\em good} if  $\mu(A\cap \tilde{A}_\Psi)\geq (1-\epsilon)\mu(A)$ and \eqref{eq:ada} is satisfied for $A$. Analogously, since $\mu(\tilde{A}_\Psi)\geq 1-\epsilon^8$, it follows that  $\epsilon^3$ a.e.\ atom is good. Fix a good atom $A$. If ${\bf D''}\subset {\bf D}$ is such that for $D\in \bf{D''}$, we have $\mu(A\cap \tilde{A}_\Psi\cap D)\leq (1-\epsilon^2)\mu(A\cap D)$, then a similar reasoning shows that $\mu(\bigcup_{D\in {\bf D''}}D)\leq \epsilon/2$. Let ${\bf D'''}=\bf{D'}\cup\bf{D''}$.
  For every $S\geq N$ we will construct a measure preserving map $\theta_S:A\to M$ such that the $S$-names of $x$ and $\theta_S(x)$ are the same except at most $\epsilon$ proportion.
 For this aim we use Proposition \ref{prop:match}. For a good atom $A$ and $D\in {\bf D } \setminus {\bf D'''}$, let $\tilde{\theta}_S=\theta_{S,D}:A\cap D\to D$ be the $\epsilon/50$ measure preserving map as in Proposition \ref{prop:match} so that \eqref{eq:aas} holds. Since $D\in {\bf D}\setminus {\bf D}'''$ it follows that $\frac{\mu(A\cap \tilde{A}_\Psi\cap D)}{\mu(A\cap D)}\in(1-\epsilon^2,1]$. Therefore $\tilde{\theta}_S$ descends to an $\epsilon/40$ measure preserving map $\theta_S:A\cap D\cap \tilde{A}_\Psi\to D$ (recall that we need to rescale the measure on $A\cap D\cap \tilde{A}_{\Psi}$).  
 This by the definition and \eqref{eq:ada} means that for every $B\subset A\cap D\cap \tilde{A}_{\Psi}$, we have 
\be\label{be:al}
\frac{\mu(A)\mu(\theta_{S,D}(B))}{\mu(A\cap D\cap \tilde{A}_{\Psi}\cap B)}=\frac{\mu(\theta_{S,D}(B))\mu(A\cap D\cap \tilde{A}_{\Psi})}{\mu(D)\mu(A\cap D\cap\tilde{A}_{\Psi} B)}\frac{\mu(A)\mu(D)}{\mu(A\cap D\cap \tilde{A}_{\Psi})}\in (1-\epsilon/10,1+\epsilon/10),
 \ee
 since both fractions are $\epsilon/30$ close to $1$ -- the first one since $\theta_{S,D}$ is $\epsilon/40$ measure preserving and the second by \eqref{eq:ada} and the fact that $\frac{\mu(A\cap \tilde{A}_\Psi\cap D)}{\mu(A\cap D)}\in(1-\epsilon^2,1]$.
  We then naturally extend the maps $\theta_{S,D}$ to a map 
$$
 \theta_S:A\cap\tilde{A}_\Psi\cap\bigcup_{D\in {\bf D}\setminus {\bf D'''}}D\to \bigcup_{D\in {\bf D}\setminus {\bf D'''}}D.
$$
We will use \eqref{be:al} to show that $\theta_S$ is $\epsilon/5$ measure preserving. Take  $B\in A\cap\tilde{A}_\Psi\cap\bigcup_{D\in {\bf D}\setminus {\bf D'''}}D$, and let $B_j:=A\cap\tilde{A}_\Psi\cap D_j$, where $j$ is an indexing of the set ${\bf D}\setminus {\bf D'''}$. 

Then $B=\bigcup_{j} B_j$ and, by \eqref{be:al}, we have 
$$
\frac{\mu(\theta_S(B))}{\mu_{|A}(B)}=\mu(A)\frac{\sum_{j}\mu(\theta_S(B_j))}{\sum_{j}\mu(A\cap B_j)}= \mu(A)\frac{\sum_{j}\mu(\theta_S(B_j))}{\sum_{j}\mu(A\cap D\cap \tilde{A}_{\Psi}\cap B_j)}\in (1-\epsilon/5,1+\epsilon/5).
$$
This shows that $\theta_S$ is $\epsilon/5$ measure preserving on 
$Dom(\theta_S):=A\cap\tilde{A}_\Psi\cap\bigcup_{D\in {\bf D}\setminus {\bf D'''}}D$. Then, by the above choice of $A$, it follows that
$\mu(Dom(\theta_S))\geq (1-\epsilon)\mu(A)$. Notice that for $x\in A'_\Psi$, \eqref{eq:aas} and \eqref{eq:betk} imply that the $S$ names of $x$ and $\theta_Sx$ match up to en error at most $\epsilon S$. We define $\theta_S$  in an arbitrary way on
$M\setminus Dom(\theta_S)$ to get an $\epsilon/2$ -measure preserving map $\theta_S:A \to M$. This finishes the proof.
\end{proof}

So it remains to prove Proposition \ref{prop:match}. We will do it in the next subsection
\subsection{Proof of Proposition \ref{prop:match}}
The following proposition is the most important part of the proof. It is a generalization of the ,,Main Lemma'' in \cite{OrnsteinWeiss}.
Recall that for $z\in M$, $W^u(\xi,z)=C(\xi,G^+)z$. For a set $\exp(S)z\in M$, where $S\in \mf g^+$, we will use the following notation $Leb^+(\exp(S)z)=Leb^+(S)$, where $Leb^+$ denotes the $p$ dimensional Lebesgue measure on $\mf g^+$.

We have the following:
\begin{proposition}[Main Proposition]\label{mainlemma} Let $z,z'\in K_\epsilon$ (see ${\bf d3}$ in Lemma \ref{lem:ext}). There exists $N_5$ such that for every $S\geq N_5$ there exists an $\epsilon/100$ - $Leb^+$ measure preserving map $\tilde{\theta}:=\tilde{\theta}_{z,z',S}:W^u(\epsilon'',z)\to W^u(\epsilon'',z')$ such that  for every $x\in W^u(\epsilon'',z)$, we have
\be\label{eq:distham}
d_M(\Psi^sx, \Psi^s\tilde{\theta}x)\leq \epsilon \text{ for  } (1-\epsilon^2)\text{ proportion  of }s\in [0,S].
\ee
\end{proposition}
Before we prove Proposition \ref{mainlemma} let us show how it implies Proposition \ref{prop:match}:
\begin{proof}[Proof of Proposition \ref{prop:match}:]
In the proof we will use that $\mu$ is $\epsilon^2$ close to $Leb^+\times Leb^0\times Leb^-$ (see (iii) in  Lemma \ref{lema:Jac}) around $e\in G$.
Fix $D\in  {\bf D}$ and $A\in \bigvee_{N_5}^N\Psi^i(\cP)$ satisfying the assumptions of Proposition \ref{prop:match} ($N_5$ to be specified below). Recall that
$$
D=C(\epsilon'',G^+)\times C(\epsilon'',G^0\times G^-)x_D\subset M.
$$
 We first use Lemma \ref{lem:intersect} to find a set $E\subset A$ such that $\mu(E)\geq (1-\epsilon^{6p})\mu(A)$ and $E$ intersects $D$ in an $\mf g^+$ tubular subset. 
 Since $E$ intersects $D$ in an $\mf g^+$ tubular subset it follows that
$$
E=C(\epsilon'',G^+)\times [E\cap C(\epsilon'',G^0\times G^-)]x_D.
$$

Let $\tilde{E}\subset E\cap C(\epsilon'',G^0\times G^-)$ be such that for every $x\in \tilde{E}$ we have  $Leb^+(C(\epsilon'',G^+)x\cap A_\Psi)\geq 1-\epsilon^{3p}$. By Fubini's theorem it follows that
$$
(Leb^{0}\times Leb^-)(\tilde{E})\geq (1-\epsilon^{3})(Leb^{0}\times Leb^-)(E\cap C(\epsilon'',G^0\times G^-))
$$
In what follows we will construct $\epsilon$- measure preserving maps on several domains. We always consider the induced measures on the domains considered.
We will first construct a map $\zeta:\tilde{E}\cap C(\epsilon'',G^0\times G^-)\to C(\epsilon'',G^0\times G^-)$ such that $\zeta$ is $\epsilon^{6p}$- $Leb^{0}\times Leb^-$ measure preserving. We will work in the Lie algebra $\mf g^0\oplus \mf g^-$. Notice that $\tilde{E}\cap C(\epsilon'',G^0\times G^-)=\exp(E'')$ for some measurable subset $E''\subset \mf g^0\oplus\mf g^-$. Let $\zeta_0:(E'',Leb^{0}\times Leb^-_{|_{E''}})\to \left(C(\epsilon'',\mf g^0\oplus\mf g^-),Leb^{0}\times Leb^-_{|_{C(\epsilon'',\mf g^0\oplus\mf g^-)}}\right)$ be an $\epsilon^{6}$ measure preserving map. Notice that such map always exists since $E''$ is a subset of Euclidean space, so we can approximate $E''$ by  disjoint cubes $\{Cub^E_i\}_{i=1}^{L(\epsilon)}$ (up to $\epsilon^{6}$) and then divide (up to $\epsilon^{6}$) the cube $C(\epsilon'',\mf g^0\oplus\mf g^-)$ into disjoint subcubes  $\{Cub_i\}_{i=1}^{L(\epsilon)}$, and set $Cub^E_i$ affinely on $Cub_i$ for $i=1,\ldots, L(\epsilon)$.
Let then  $\zeta:\tilde{E}x_D\to C(\epsilon'',G^0\times G^-)x_D$ be the lift of the map $\zeta_0$ (then $\zeta$ is also $\epsilon^{6}$ $Leb^0\times Leb^-$ measure preserving). 
We now define the map
$\theta_S:C(\epsilon'',G^+)\times \tilde{E}x_D\to C(\epsilon'',G^+)\times C(\epsilon'',G^0\times G^-)x_D$
\be\label{eq:amdn}
\theta_S(x,z)=\tilde{\theta}_{z, \zeta(z),S}(x),
\ee
(we extend $\theta_S$ to $E$ in an arbitrary way on $\tilde{E}^c$). Notice that by Fubini's theorem, $\theta_S$ is $\epsilon^2$ measure preserving. Therefore, on  $C(\delta,G^+)\times \tilde{E}_0x_D$, we have that $z, \zeta(z)$ satisfy the assumptions of Proposition \ref{mainlemma}. So by the definition of $\theta_S$ (restricted to $C(\delta,G^+)\times \tilde{E}_0x_D$ ) it follows that \eqref{eq:distham} holds which implies that \eqref{eq:aas} holds. This finishes the proof of Proposition \ref{prop:match}.

\end{proof}

So it only remains to prove Proposition \ref{mainlemma}. We will do this in a separate subsection.

\subsection{Proof of Proposition \ref{mainlemma}}\label{secmainlem}
In this section we will give the proof of Proposition \ref{mainlemma}. Before we do that we will give an outline of the proof to explain the main ideas.
\\
\textbf{Outline of the proof.} The main idea is to use effective equidistribution on cubes (see Lemma \ref{eq:distcube}), i.e.\ for a fixed time $S$ let $k$ be such that $S$ is of order $\log k$. Let $z,z'\in K_\epsilon$. Then for a fixed $j\in \{1,\ldots, N_k\}$ we consider the set of all points from $x\in W^u(\epsilon'',z)$ such that  $\Psi^Nx\in D_j$. By the invariance of foliations it follows that this set is a  union of (exponentially small) cubes in $W^u(\epsilon'',z)$ and each cube is mapped onto the unstable piece in $D_j$ as Figure 1 shows. The same holds for $z'$. Then there is a natural mapping $\theta$ between an exponentially small cube in $W^u(\epsilon'',z)$ and $W^u(\epsilon'',z')$, see Figure 1. It is the composition of the maps on the unstable pieces with the center-stable holonomy map in $D_j$. Notice that since the size of $D_j$ is a large power of $\frac{1}{\log k}$ (and the time $S$ is of order $\log k$), it follows that (on the two small cubes) $S$-orbits of points $\psi^Nx$ and $\psi^N\theta x$ are $\epsilon$ close -- indeed, they lie on the same central stable piece (because of the holonomy map) and the growth on the central stable piece is at most  polynomial in $S$ (see Lemma \ref{lem:quasi}).
\begin{figure}[h]
\begin{center}
\includegraphics[scale=0.50]{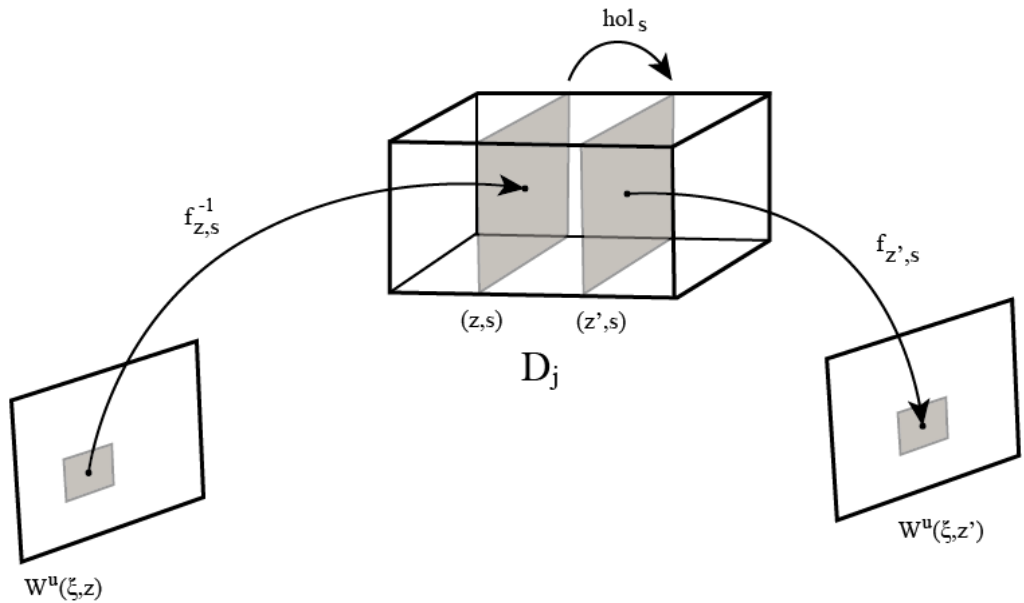}
		$\quad\quad\quad\quad$
		\caption{The matching function between two small cubes in $W^u(\epsilon'',z)$ and $W^u(\epsilon'',z')$}
		\end{center}
	\end{figure}

 The key observation is that for each $j\in \{1,\ldots,N_k\}$ the number of small cubes in $W^u(\epsilon'',z)$ and $W^u(\epsilon'',z')$ is the same up to an error of order $\epsilon^2$, see Figure 2. This is a consequence of quantitative equidistribution on $D_j$ (see Lemma \ref{eq:distcube}).  Moreover, if we consider the union over all $j$ of all small cubes in $W^u(\epsilon'',z)$ (or $W^u(\epsilon'',z')$), then again by Lemma \ref{eq:distcube}, they cover $W^u(\epsilon'',z)$ (or $W^u(\epsilon'',z')$) up to a proportion of $\epsilon^2$. Extending $\theta$ in a natural way over all $j\in \{1,\ldots,N_k\}$ yields a matching between $W^u(\epsilon'',z)$ and $W^u(\epsilon'',z')$.
\begin{figure}[h]
\begin{center}
\includegraphics[scale=0.38]{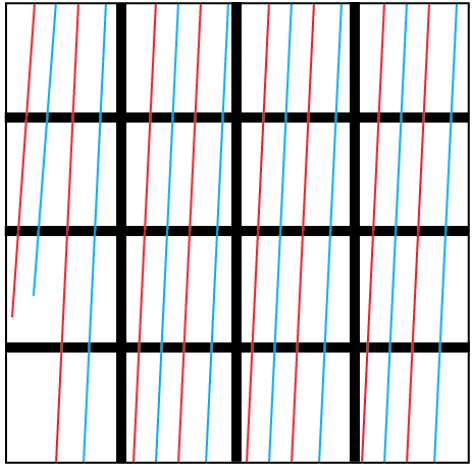}
		$\quad\quad\quad\quad$
		\caption{Images of $W^u(\epsilon'',z)$ (red) and $W^u(\epsilon'',z')$ (blue) under $\Psi^N$ are equidistributed in the cubes $D_j$ (white squares) up to an $\epsilon$ proportion.}
		\end{center}
	\end{figure}

\newpage
\begin{proof}[Proof of Proposition \ref{mainlemma}] Fix $S\geq \max(\epsilon^{-10}N_3, \epsilon^{-10}k_\epsilon)$.  Let $N=[\epsilon^3 S]$. Let $k\in \N$ be such that $N\in [\log k,\log(k+1)]$. Let $N_k$ and  ${\bf D}_j=D(\delta_k,x_j)$, $j=1,\ldots, N_k$ be the family of cubes from Lemma \ref{eq:distcube} with $\delta_k=\frac{1}{(\log k)^{r'}}$. Fix $j\in 1,\ldots, N_k$. By Lemma \ref{eq:distcube} with $\xi=\epsilon''$, we have 

for $w\in\{z,z'\}\subset K_{\epsilon}$, we have
\be\label{ee}
\left|\frac{Leb^+\left(\{P\in C(\epsilon'',\mf g^+)\;:\; \Psi^N(\exp(P)w)\in {\bf D}_j  \}\right)}{Leb^+(C(\epsilon'',\mf g^+))\mu({\bf D_j})}-1\right|\leq \epsilon^2.
\ee
Let $C_N\subset C(\delta,\mf g^+)$ be given by
\be\label{def:CN}
C_N:=C((1-\lambda^{-N})\delta,\mf g^+),
\ee
where $\lambda:=\frac{1}{2}\min_{|\lambda| >1}|\lambda|$. Notice that
\be\label{mesrt}
Leb^+(C(\delta,\mf g^+)\setminus C_N)\leq [1-(1-\lambda^{-N})^p]Leb^+(C(\delta,\mf g^+)).
\ee
Recall that $j\in 1,\ldots,N_k$ is now fixed.  Take any $P_j\in C_N$ such that  $\Psi^N(\exp(P_j)z)\in {\bf D}_j$, i.e.\
$$\Psi^N\exp(P_j)z=\exp(V_j)\exp(W_j)x_j$$
 for some $V_j\in C(\delta_k,\mf g^+)$ and $W_j\in C(\delta_k,\mf g^0\oplus \mf g^-)$.

 We now define a map $f_{P_j}:C(\delta_k,\mf g^+)\to C(\epsilon'',\mf g^+)$ which sends $V_j$ to $P_j$ and such that $f\left(C(\delta_k,\mf g^+)\right)\subset C(\epsilon'',\mf g^+)$ (in fact, as we will see below, the image of  $C(\delta_k,\mf g^+)$ will be an exponentially small set in $C(\epsilon'',\mf g^+)$). Let $R_Q:G \to G$, $R_Q(h):=h\exp(Q)$ and define\footnote{Notice that if $P=V=0 \in \mf g^+$, then $f_{P_j}(U)=Ad(g)^{-N}(U)$, which is the adjoint action on $U$. In our case we want the image of $V_j$ to be $P_j$ and therefore we need to shift twice by $\exp(P_j)$ and $\exp(-V_j)$.}
$$
f_{P_j}(U):=\left[\log_{alg}\circ R_{P_j}\circ\exp \circ Ad(g)^{-N}\circ\log_{alg}\circ R_{-V_j}\circ\exp\right](U).
$$
Notice that the inverse of this map is given by
\be\label{eq:inv}
f^{-1}_{P_j}(U):=\left[\log_{alg}\circ R_{V_j}\circ \exp\circ Ad(g)^{N}\circ\log_{alg}\circ R_{-P_{j}}\circ\exp\right](U)
\ee

Since $\|V_j\|,\|P_j\|\leq \epsilon''$ ( in fact, $\|V_j\|\leq \delta_k$), we have
$$
\left|\det[Jac(\log_{alg}\circ R_{P_j}\circ \exp)(\cdot)]\right|\in (1-\epsilon''^{1/3},1+\epsilon''^{1/3}).
$$
Similarly,
$$
\left|\det[Jac[(\log_{alg}\circ R_{-V_j}\circ \exp)(U)\right|\in (1-\epsilon''^{1/3},1+\epsilon''^{1/3}).
$$
Notice that in the proof we consider $Ad(g)^{-N}$ restricted to the space $\mf g^+$. To simplify the notation we use $Ad(g)^{-N}$ to denote $Ad(g)^{-N}_{|\mf g^+}$. Since $Ad(g)^{-N}$ is a linear map, we have (denoting $c_g:=\prod_{|\lambda|>1}|\lambda|\dim V_\lambda>1$
$$
\det Jac(Ad(g)^{-N}(\cdot))=|\det Ad(g)^{-N}|=c_g^{-N}.
$$
Summarizing, it follows that for $U\in C(\delta_k,\mf g^+)$, we have
\be\label{lab:jaces}
\left|\det[Jac(f_{P_j})](U)\right|\in [1-\epsilon''^{1/10},1+\epsilon''^{1/10}]c_g^{-N}.
\ee
Moreover, since $P_j\in C_N$, it follows that $f_{P_j}(C(\delta_k,\mf g^+))\subset C(\epsilon'',\mf g^+)$ (since $C(\delta_k,\mf g^+)$ is contracted exponentially by $f_{P_j}$, $f_{P_j}(V_j)=P_j$ and $P_j$ is separated from the boundary of $C(\epsilon'',\mf g^+)$ and $N$ is large enough, see \eqref{def:CN}).

Therefore, by the construction of the function $f_{P_j}$ and since $Ad(g)^N$ preserves the spaces $C(\cdot,\mf g^+)$ and the spaces $\mf g^+$ and $\mf g^0\oplus \mf g^-$ are transversal,  it follows that there exists $\{P^z_{j,s}\}_{s=1}^{\ell_j(z)}\in C_N$ such that  
\be\label{eq:coord}
\Psi^N\exp(P^z_{j,s})z=\exp(V^z_{j,s})\exp(W^z_{j,s})x_j
\ee
 for $s=1,\ldots,\ell_j(z)$ and
\begin{multline}\label{eq:znmes}
C_N \cap \{P\in C(\epsilon'',\mf g^+)\;:\; \Psi^N\exp(P)z\in {\bf D}_j  \}\subset \bigcup_{s=1}^{\ell_j(z)}f_{P^z_{j,s}}\left(C(\delta_k,\mf g^+)\right)\subset \\
\{P\in C(\epsilon'',\mf g^+)\;:\; \Psi^N\exp(P)z\in {\bf D}_j  \}.
\end{multline}
We get analogous sets $\{P^{z'}_{j,s}\}_{s=1}^{\ell_j(z')}$ for $z'$. Notice that by the bounds on the Jacobian, \eqref{lab:jaces}, it follows that
\be\label{maswq}
Leb^+\left(\bigcup_{s=1}^{\ell_j(z)}f_{P^z_{j,s}}\left(C(\delta_k,\mf g^+)\right)\right)\in \ell_j(z)[1-\epsilon''^{1/10},1-\epsilon''^{1/10}]c_g^{-N}.
\ee

Moreover, since $(1-(1-\lambda^{-N})^p)\leq \epsilon^2\mu({\bf D}_j)$ and $N\geq \log k$, by \eqref{ee} and \eqref{mesrt} it follows that
\begin{multline}\label{eq:afe1}
Leb^+\left(C_N \cap \{P\in C(\epsilon'',\mf g^+)\;:\; \Psi^N\exp(P)z\in {\bf D}_j \}\right)\geq \\
(1-\epsilon^2)Leb(C(\epsilon'',\mf g^+)\mu({\bf D}_j)- Leb^+(C(\epsilon'',\mf g^+)\setminus C_N)\geq (1-\epsilon^{3/2})Leb^+(C(\epsilon'',\mf g^+)\mu({\bf D}_j)
\end{multline}
This together with \eqref{eq:znmes}  and \eqref{maswq} implies that  (since we assume that $\epsilon''<\epsilon^{100}$)
\be\label{zest}
\ell_j(z)\in [1-\epsilon^{4/3},1+\epsilon^{4/3}] c_g^{N}Leb^+(C(\epsilon'',\mf g^+))\mu({\bf D}_j)
\ee
The same estimate holds for $\ell_j(z')$. Let $\ell_j:=\min(\ell_j(z),\ell_j(z'))$, $C(j,z,s)=f_{P^z_{j,s}}\left(C(\delta_k,\mf g^+)\right)$ and define
\be\label{eq:asdt}
 C(j,z):=\bigcup_{s=1}^{\ell_j}C(j,z,s)\subset C(\epsilon'',\mf g^+).
\ee

Since $\exp(W^{z}_{j,s})\in C(\delta_k,\mf g^0\oplus \mf g^-)$ and $\exp(W^{z'}_{j,s})\in C(\delta_k,\mf g^0\oplus \mf g^-)$ (see \eqref{eq:coord}) it follows that $\log_{alg}\left(\exp(W^{z}_{j,s})\exp(-W^{z'}_{j,s})\right)$ exists. Let $hol^{z,z'}_{j,s}:\mf g^+\oplus\{\log_{alg}\left(\exp(W^{z}_{j,s})\exp(-W^{z'}_{j,s})\right)\}\to \mf g^+$ be the stable holonomy defined in  \eqref{eq:holonomymap} (since $z,z'\in M$ are fixed throughout the proof, we drop it from the notation of the holonomy map). By \eqref{eq:holonomymap} and Lemma \ref{lem:hol} it follows that for $R\in C(\delta_k,\mf g^+)$, we have 
$$
\exp(R)\exp(W^{z}_{j,s})\exp(-W^{z'}_{j,s})=\exp(W''_{j,s})\exp(hol_{j,s}(R)),
$$
where $W''_{j,s}=hol_{uc}^{z,z',j,s}\left(R,\log_{alg}\left(\exp(W^{z}_{j,s})\exp(-W^{z'}_{j,s})\right)\right)\in \mf g^0\oplus\mf g^-$. This transforms to 
\be\label{neweq:hol}
\exp(R)\exp(W^{z}_{j,s})=\exp(W''_{j,s})\exp(hol_{j,s}(R))\exp(W^{z'}_{j,s})
\ee

Moreover, by Lemma \ref{lem:hol} it follows that 
$$
 \|W''_{j,s}-\log_{alg}\left(\exp(W^{z}_{j,s})\exp(-W^{z'}_{j,s})\right)\|\leq \kappa_2\|R\|\|\log_{alg}\left(\exp(W^{z}_{j,s})\exp(-W^{z'}_{j,s})\right)\|.
 $$
This together with the fact that $R\in C(\delta_k,\mf g^+)$, $W^{z}_{j,s}\in C(\delta_k,\mf g^0\oplus\mf g^-)$ and $W^{z'}_{j,s}\in C(\delta_k,\mf g^0\oplus\mf g^-)$ implies, using $(ii)$ in Lemma \ref{lema:Jac}, that 
\be\label{normW''}
 \|W''_{j,s}\|\leq 4\delta_k, \text{ and }W''_{j,s}\in \mf g^0\oplus\mf g^-.
\ee
By the same arguments and Lemma \ref{lem:hol}, we also have
 that if $\|R\|<(1-\epsilon''^{100})\delta_k$, then
 $\|hol_{j,s}(R)\|\leq \delta_k$. By throwing away a subset of measure $\epsilon^{100}Leb^+(C(\delta_k,\mf g^+)$,
 we restrict the domain of $hol_{j,s}$  so that the image is inside $C(\delta_k,\mf g^+)$.
Notice that by Lemma \ref{lem:hol}, we have  $|\det Jac\;hol_{j,s}|\in (1-\epsilon^4,1+\epsilon^4)$.
We define the following {\em matching map}  $\theta_{j,s}$ between $C(j,z,s)$ and $C(j,z',s')$.
\be\label{eq:tildaa}
\theta_{j,s}(R)= f_{P^{z'}_{j,s}}\circ  hol_{j,s}\circ f_{P^z_{j,s}}^{-1}(R)
\ee
We then naturally, using \eqref{eq:asdt}, extend the map to a map $\theta_j:C(j,z)\to C(j,z')$ and finally we extend it to a map $\theta_{z,z'}:\bigcup_{j=1}^{N_k}C(j,z)\to \bigcup_{j=1}^{N_k}C(j,z)$, since the sets $\{C(j,z)\}_{j=1}^{N_k}$ are pairwise disjoint (see the right inclusion in \eqref{eq:znmes} and recall that the sets $\{{\bf D}_j\}_{j=1}^{N_k}$ are disjoint). We will show that $\theta_{z,z'}$ is $\epsilon/1000$ measure preserving. Notice first that by \eqref{zest} it follows that $|\ell_j-\ell_j(z)|<2\epsilon^{4/3}\ell_j(z)$. Therefore and using \eqref{eq:afe1} and \eqref{eq:znmes}, we get that  the domain of $\theta_{j}$ has measure at least
\begin{multline*}
Leb^+\left(\bigcup_{s=1}^{\ell_j}C(j,z,s)\right)\geq Leb^{+}\left(\bigcup_{s=1}^{\ell_j(z)}C(j,z,s)\right)-Leb^{+}\left(\bigcup_{s=\ell_j}^{\ell_j(z)}C(j,z,s)\right)\\
\geq (1-\epsilon^{3/2}-2\epsilon^{4/3})Leb^+(C(\delta,\mf g^+))\mu({\bf D}_j).
\end{multline*}

Therefore the domain of  $\theta_{z,z'}$ has measure at least
$$
Leb^+(\bigcup_{j=1}^{N_k}C(j,z))\geq (1-\epsilon^{3/2}-2\epsilon^{4/3})Leb^+(C(\delta,\mf g^+))\mu(\bigcup_{j=1}^{N_k}{\bf D}_j)\geq (1-\epsilon^{11/10})Leb^+(C(\delta,\mf g^+)),
$$
since by definition of the family ${\bf D}_j$, we have $\mu(\bigcup_{j=1}^{N_k}{\bf D}_j)\geq (1-\epsilon^2)$ (see ${\bf d2}$ for $\epsilon'=\epsilon^2$ in Lemma \ref{lem:ext}). Therefore $\theta_{z,z'}$ is defined on a proportion $1-\epsilon^{11/10}$ of $C(\epsilon'',\mf g^+)$. Finally, we define
\be\label{tilea}
\tilde{\theta}_{z,z',S}(\exp(R)z)=\exp(\theta_{z,z'}(R))z',
\ee
We need to show that $\tilde{\theta}_{z,z',S}$ is $\epsilon/1000$- $Leb^+$ measure preserving and that \eqref{eq:distham} holds. To show the first part, it is enough to show that $\theta_{j,s}$ (see \eqref{eq:tildaa}) is $\epsilon/1000$-$Leb^+$ measure preserving. Note that  the Jacobian of $\exp$, $\log_{alg}$ is $\epsilon^5$ close to $1$. Moreover, the Jacobian of $hol_s$ is also $\epsilon^5$ close to $1$. Finally by \eqref{lab:jaces}, we get that$$
|\det[Jac(f_{P^{z'}_{j,s}})]\det[Jac(f^{-1}_{P^{z}_{j,s}})]|\in [(1-\epsilon^2)^2,(1+\epsilon^2)^2].
$$
So by the chain rule we have that
$$
|\det[Jac(\theta_{j,s})]|\in [1-\epsilon^2, 1+\epsilon^2].
$$
This shows that $\theta_{j,s}$ (and hence $\theta_{z,z',S}$) is $\epsilon/1000$ measure preserving. It remains to show that $\tilde{\theta}_{z,z',S}$ satisfies \eqref{eq:distham}. Take $x\in W^u(\epsilon'',z)$, $x=\exp(U_x)z$ and such that $\theta_{z,z',S}$ is defined on $x$ (notice that the measure of such points is at least $(1-\epsilon^{11/10})Leb^+(C(\epsilon'',\mf g^+))$).
Let $j\in \{1,\ldots,N_k\}$ and $s\in \{1,\ldots \ell\}$ be such that $x\in C(j,z,s)$.  Then by \eqref{tilea} and the definition of $\theta_{z,z'}$, $\tilde{\theta}_{z,z',S}(\exp(U_x)z)=\exp(\theta_{j,s}(U_x))z'$. Denote $\theta_{j,s}(U_x)=\tilde{U}_x$.

We have
\begin{multline}\label{niceform}
\Psi^N\exp(U_x)z=\\
\Psi^N\exp(\log_{alg}\left[\exp(U_x)\exp(-P^z_{j,s})\right])\Psi^{-N}\Psi^N\exp(P^z_{j,s})z=\\\exp\left(Ad(g)^{N}\circ\log_{alg}\circ R_{-P^z_{j,s}}(\exp(U_x)\right)\exp(V^z_{j,s})\exp(W^z_{j,s})x_j=\\
\exp\left(\log_{alg}\left[ \exp\left(Ad(g)^{N}\circ\log_{alg}\circ R_{-P^z_{j,s}}(\exp(U_x)\right)\exp(V^z_{j,s})\right]\right)\exp(W^z_{j,s})x_j.
\end{multline}
Hence the unstable component of $\Psi^N\exp(U_x)z$ equals (see \eqref{eq:inv})
\be\label{eds}
\log_{alg}\circ R_{V^{z}_{j,s}}\circ \exp\circ Ad(g)^{N}\circ\log_{alg}\circ R_{-P^{z}_{j,s}}\circ \exp](U_x)=f^{-1}_{P_{j,s}^{z}}(U_x)
\ee
Analogously,
\begin{multline}\label{niceform2}
\Psi^N\exp(\tilde{U}_x)z'=\\
\Psi^N\exp(\log_{alg}\left[\exp(\tilde{U}_x)\exp(-P^{z'}_{j,s})\right])\Psi^{-N}\Psi^N\exp(P^{z'}_{j,s})z=\\\exp\left(Ad(g)^{N}\circ\log_{alg}\circ R_{-P^{z'}_{j,s}}(\exp(\tilde{U}_x)\right)\exp(V^{z'}_{j,s})\exp(W^{z'}_{j,s})x_j=\\
\exp\left(\log_{alg}\left[ \exp\left(Ad(g)^{N}\circ\log_{alg}\circ R_{-P^{z'}_{j,s}}(\exp(\tilde{U}_x)\right)\exp(V^{z'}_{j,s})\right]\right)\exp(W^{z'}_{j,s})x_j.
\end{multline}
Hence the unstable component of $\Psi^N\exp(\tilde{U}_x)z$ equals (see \eqref{eq:inv})
\be\label{eds2}
\log_{alg}\circ R_{V^{z'}_{j,s}}\circ \exp\circ Ad(g)^{N}\circ\log_{alg}\circ R_{-P^{z'}_{j,s}}\circ \exp](\tilde{U}_x)=f^{-1}_{P_{j,s}^{z'}}(\tilde{U}_x).
\ee
Recall that  $\tilde{U}_x=\theta_{j,s}(U_x)$. So by \eqref{eq:tildaa} it follows that
\be\label{eq:adf}
f^{-1}_{P_s^{z'}}(\tilde{U}_x)= hol_{j,s}\circ f^{-1}_{P_s^{z}}(U_x).
\ee
Therefore, by \eqref{niceform} and \eqref{eds}, we get 
\be\label{expz}
\Psi^N\exp(U_x)z=\exp(f^{-1}_{P_{j,s}^{z}}(U_x))\exp(W_{j,s}^z)x_j
\ee
and by \eqref{niceform2}, \eqref{eds2} and\eqref{eq:adf}, we have
\be\label{expz'}
\Psi^N\exp(\tilde{U}_x)z'=\exp\left(hol_{j,s}\left(f^{-1}_{P_{j,s}^{z}}(U_x)\right)\right)\exp(W_{j,s}^{z'})x_j.
\ee
Moreover by \eqref{neweq:hol} for $R=f^{-1}_{P_{j,s}^{z}}$, for some $W''_{j,s}\in C(4\delta_k,\mf g^0\oplus\mf g^-)$ (see \eqref{normW''}) we have 
\be\label{expzzz}
 \exp\left(hol_{j,s}\left(f^{-1}_{P_{j,s}^{z}}(U_x)\right)\right)\exp(W_{j,s}^{z'})x_j=\exp(-W''_{j,s})\exp(f^{-1}_{P_{j,s}^{z}}(U_x))\exp(W_{j,s}^z)x_j.
 \ee
It follows that \eqref{expz}, \eqref{expz'} and \eqref{expzzz} give 
$$
\Psi^N\exp(U_x)z=\exp(W''_{j,s})(\Psi^N\exp(\tilde{U}_x)z').
$$
Therefore, for every $s\in [0,S-N]$, we have
$$
d_M(\Psi^{s+N}\exp(U_x)z,\Psi^{s+N}\exp(\tilde{U}_x)z')\leq d_G(\Psi^s(\exp(W''_{j,s})),\Psi^s(e))\leq \epsilon^2,
$$
since $s\leq S-N\leq S\leq \epsilon^{-3}N\leq \epsilon^{-3}\log (k+1)$, $W''_{j,s}\in C(4\delta_k, \mf g^0\oplus \mf g^-)$, $\delta_k=(\log k)^{-r'}$ and we use Lemma \ref{lem:quasi} remembering that $r'>r_0+1$ and $\epsilon^{-3}$ is small compared to $\log k$. This finishes the proof of Proposition \ref{mainlemma}.
\end{proof}
\section{Appendix: Proof of Lemma \ref{eq:distcube}}
Let $m_+$ denote the Haar measure on the expanding subgroup $G^+$. We will use the following modification of Proposition 2.4.8. in \cite{KlMa}:
\begin{proposition}\label{prop:klma}
There exists $\lambda, d, d', q>0$ such that for any $f\in C_c^{\infty}(G^+)$, $supp f\subset B^+(e,1)$, any $\phi\in C_c^\infty(M)$ and any compact set $K\subset M$, there exists a constant $C(K)$ such that for every $n\geq 0$ and every $x\in K$, we have 
\be\label{eq:ads}
\left|\int_{G^+}f(g)\phi(\Psi^n(gx))dm_+-\int_{G^+}f dm_+\int_M \phi d\mu\right|\leq C(K)[\max_{x\in M}\nabla \phi]^{d'}  \mathcal{S}_d(f)\mathcal{S}_d(\phi)t^qe^{-\lambda n}.
\ee
\end{proposition}
\begin{proof}The proof of the above proposition is analogous to the proof of Proposition 2.4.8. in \cite{KlMa}. There the authors consider a flow $(g_t)$ instead of an automorphism $\Psi$ but the important property used is the exponential mixing of $(g_t)$, which also holds in our case  for $\Psi$ (see Lemma \ref{lem:fff}). Moreover, the statement of  Proposition 2.4.8. is weaker, i.e.\ the authors claim the existence of a constant $C(f,\phi,K)$ such that \eqref{eq:ads} holds with $C(f,\phi,K)$ instead of $C(K)[\max_{x\in \Omega}\nabla \phi]^{d'}\mathcal{S}_d(f)\mathcal{S}_d(\phi)$. But the proof of Proposition 2.4.8. 
gives an explicit constant. First, by changing the constant (but only depending on $K$), one can assume that the projection maps $\pi_x:G\to M$, $\pi_x(g)=gx$ are injective on $supp\, f$ for every $x\in K$. If this is the case,  equation 2.4.8 in \cite{KlMa} and the estimates below explicitely give the constant $C(f,\phi,K)$ in the last line of the proof.
\end{proof}

Then Lemma \ref{eq:distcube} follows from approximating characteristic functions of $D(\delta_k,x_j)$ with smooth compactly supported functions, with good control on the Sobolev norms. Notice also, that close to $e\in G$ the measures $Leb^+$ and $m_+$ are close (see $(iii)$ in Lemma \ref{lema:Jac}). We will prove the following lemma:

 For $\epsilon>0$ and $B\subset M$, we denote the $\epsilon$ neighborhood $V_{\epsilon}(B):=\{x\in M\;:\; d_M(x,B)<\epsilon\}$. Analogously we define $\epsilon$ neighborhood  (in $G$) of a set $B\subset G$. Let $r'=100r_0^2$, where $r_0$ comes from Lemma \ref{lem:quasi}. For $\epsilon>0$ let $K_\epsilon$ be the set from ${\bf d3}$ in Lemma \ref{lem:ext}. Let $sys(\epsilon)=sys(K_{\epsilon})$ (see Lemma \ref{lem:systole}) and let 
$$
s(\epsilon):=\min((sys(\epsilon))^2,\epsilon^4).
$$
First we state the following approximation lemma.
\begin{lemma}\label{gooda} Let $\delta_k:=(\log k)^{-r'}$ and denote $\tilde{C}_k:=C(\delta_k,G^{+})\times C(\delta_k,G^{0}\times G^-)\subset G$. For every $\epsilon>0$ there exists $k_\epsilon$ such that for every $k\geq k_\epsilon$ there are  function $\tilde{\phi}^+_k,\tilde{\phi}^-_k\in C_c^{\infty}(G)$ such that the following holds:
\begin{itemize}
\item[$V0.$] $\tilde{\phi}^-_k\leq \chi_{\tilde{C}_k}\leq \tilde{\phi}_k^{+}$
\item[$V1.$] for $w\in\{+,-\}$,  $\int_G|\tilde{\phi}_k^w-\chi_{\tilde{C}_k}|d\mu_G<\epsilon^3\mu_G(\tilde{C}_k)$ ;
\item[$V2.$] for $w\in \{+,-\}$, $\tilde{\phi}^w_k\equiv 0$ outside $V_{s(\epsilon)}(\tilde{C}_k)$;
\item[$V3.$] there exists a global constant $R=R(d)>0$ such that $\mathcal{S}_d(\tilde{\phi}^w)\leq (\log k)^{R}$ for $w\in \{+,-\}$.
\end{itemize}
\end{lemma}
\begin{proof}
We will only show how to construct $\tilde{\phi}_k^+$, the construction of $\tilde{\phi}_k^-$ is analogous.

Let $\eta^\ell_\epsilon\in C^{\infty}(\R^\ell)$ be a {\em smoothing function} in $\R^\ell$, such that
$$
\eta^\ell_\epsilon({\bf v})=1 \text{ if } {\bf v}\in [-1,1]^\ell\l\l \text{, }\;\eta^\ell_\epsilon({\bf v})=0 \text{ if } {\bf v}\notin [-1-s(\epsilon)^3,1+s(\epsilon)^3]^\ell
$$
and the derivatives satisfy: for every $m\in \N$ and for every $i\in \{1,\ldots,\ell\}$,  
$$\left|\frac{\partial^m(\eta^\ell_\epsilon)}{\partial^m x_i}\right|\leq s(\epsilon)^{-6\ell m}
$$ Let then $\eta^\ell_{k,\epsilon}(x):=\eta_\epsilon^\ell\left(\frac{1}{\delta_k}x\right)$ (since $\epsilon$ is fixed, we will drop it from the notation). If we pick $k\geq k_\epsilon$, where $k_\epsilon$ is such that  $\delta^{-1}_{k_{\epsilon}}\geq \log k_\epsilon>s(\epsilon)^{6\ell}$, we have
\be\label{eq:asdfgh}
\left|\frac{\partial^m(\eta^\ell_\epsilon)}{\partial^m x_i}\right|\leq \delta_k^{-12\ell m}, \text{ for every }m\in \N \text{ and } i\in\{1,\ldots,\ell\}.
\ee
Let $\ell_1=\dim \mf g^+$ and $\ell_2=\dim \mf g^0+\dim \mf g^-$.
Define\footnote{For a  vector $U\in \mf g^+$ and a function $\eta:\R^{\dim \mf g^+}\to \R$, we denote $\eta(U):=\eta(a_1,\ldots a_{\dim \mf g^+})$, where $U=\sum a_i V_i$. We use analogous notation for $\mf g^0\oplus \mf g^-$.}
$$
\tilde{\phi}(\exp(U)\exp(H)):=\exp(\eta_k^{\ell_1}(U){\bf V_+})\exp(\eta_k^{\ell_2}(H){\bf V_{0,-}})
$$
if $\exp(U)\exp(H)\in V_{s(\epsilon)}(\tilde{C}_k)$ and set $\tilde{\phi}(g)=0$ otherwise. With this definition of $\tilde{\phi}$, properties $V0$ and $V2.$ hold automatically. Moreover by the bounds on the derivatives of  $\eta_k^{\ell_1}$ and $\eta_k^{\ell_2}$ (see \eqref{eq:asdfgh}) and the definition of Sobolev norms (see Section \ref{sec:sob}) we get that  there exist $R'=R'(d)>0$ and $R=rR'$ such that
$$
\mathcal{S}_d(\tilde{\phi})\leq \delta_k^{R'}=(\log k)^{rR'}\leq \log_k^{R}.
$$
this gives $V3$.
For $V1$ notice that $\tilde{f}(g)=1$ on $\tilde{C}_k$ and hence
$$
\int_G|\tilde{\phi}-\chi_{\tilde{C}_k}|d\mu(G)< \mu\left(V_{s(\epsilon)}(\tilde{C}_k)\setminus \tilde{C}_k\right)\leq \epsilon^3\mu(\tilde{C}_k),
$$
where the last inequality follows from $(iii)$ in Lemma \ref{lema:Jac} for  
$V_{s(\epsilon)}(\tilde{C}_k)$ and $\tilde{C}_k$.
This finishes the proof.

\end{proof}

Using Lemma \ref{gooda} and Proposition \ref{prop:klma} we can prove Lemma \ref{eq:distcube}.

\begin{proof} 
Let $\tilde{C}_k=C(\delta_k,G^{+})\times C(\delta_k,G^{0}\times G^-)\in G$. Then by definition,
$D\left(\delta_k,x_j\right)=\tilde{C}_kx_j\Gamma$.
Let
$$
\phi^+_j:V_{s(\epsilon)}(\tilde{C}_k)x_j\to \R,\;\;\;\;\phi^+_j(g):=\tilde{\phi}^+(gx_j^{-1}),
$$
where $\tilde{\phi}^+$ comes from Lemma \ref{gooda} for $\tilde{C}_k$. Notice that $\phi^+_j\equiv 0$ outside $V_{s(\epsilon)}(\tilde{C}_k)x_j$.
Moreover, for every $\gamma\in \Gamma\setminus\{e\}$, we have
\be\label{znb}
V_{s(\epsilon)}(\tilde{C}_k)x_j\cap V_{s(\epsilon)}(\tilde{C}_k)x_j\gamma=\emptyset.
\ee
Indeed, if $z=g_1x_j=g_2x_j\gamma$, with $g_1,g_2\in V_{s(\epsilon)}(\tilde{C}_k)$, then since $x_j\in K_{\epsilon'}$ (see ${\bf d3}$ in Lemma \ref{lem:ext}), we have
$$
0=d_G(g_1x_j,g_2x_j\gamma)\geq d(x_j\gamma x_j^{-1})-d(g_1,e)-d(g_2,e)\geq sys(\epsilon)-2s(\epsilon)-2\delta_k>0,
$$
since $s(\epsilon)<sys(\epsilon)^2$ and $k$ is large enough.
Since  $\phi^+_j\equiv 0$ outside $V_{s(\epsilon)}(\tilde{C}_k)x_j$ and by \eqref{znb}, we get that $\phi^+_j$ extends to a $\Gamma$- periodic function $\bar{\phi}^+_j$ such $\bar{\phi}^+_j=\phi^+_j$ in $V_\epsilon(\tilde{C}_k)x_j$. So $\bar{\phi}^+_j$ descends to a  $C_c^{\infty}$ function on $M$. By the construction of $\bar{\phi}^+_j$ and Lemma \ref{gooda} it follows that $V1$, $V2$, $V3$ hold with $G$, $\tilde{\phi}$, $\mu_G$ and $\tilde{C}_k$ replaced respectively by $M$, $\bar{\phi}_j$, $\mu$, $D\left(\delta_k,x_j\right)$. For $\xi\in (0,\epsilon)$ let $f_{+,\xi}\in C_c^\infty(G^+)$ be a function such  that 
\begin{itemize}
\item[$W0.$] $\chi_{\exp(C(\xi,\mf g^+))}\leq f_{+,\xi}$,
\item[$W1.$]  $\int_{G^+}|f_{+,\xi}-\chi_{\exp(C(\xi,\mf g^+))}|d\mu_+<\epsilon^3Leb^+(C(\xi,\mf g^+))$ ;
\item[$W2.$] $f_{+,\xi}\equiv 0$ on $V_{s(\epsilon)}(\exp(C(\xi,\mf g^+)))$;
\item[$W3.$] $\mathcal{S}_d(f_{+,\xi})\leq \xi^{-2d}$.
\end{itemize}

Then, since $m_+$ and $Leb^+$ are almost the same around $e$, 
\begin{multline*}
\int_{C(\xi,\mf g^+)} \chi_{D\left(\delta_k,x_j\right)}(\Psi^N\exp(H)x)dLeb^+\leq \\
\int_{C(\xi,\mf g^+)}f_{+,\xi}(\exp(H)) \bar{\phi}^{+}_j(\Psi^N\exp(H)x)dLeb^+\leq
(1+\epsilon^3)\int_{G^+}f_{+,\xi}(g)\bar{\phi}^{+}_j(\Psi^N gx)dm_+.
\end{multline*}
But by Proposition \ref{prop:klma} and Lemma \ref{gooda} and $W3.$, for $x\in K_\epsilon$ and $N\geq \log k$, we have 
\begin{multline*}
\left|\int_{G^+}f_{+,\xi}(g)\bar{\phi}^{+}_j(\Psi^N gx)dm_+-\int_{G^+}f_{+,\xi} dm_+\int_{G}\bar{\phi}^{+}_jd\mu\right|\leq\\
 C(K_\epsilon)[\max_{x\in \Omega}\nabla \bar{\phi}^{+}_j]^{d'}  \mathcal{S}_d(f_{+,\xi})\mathcal{S}_d(\bar{\phi}^{+}_j)t^qe^{-\lambda N}\leq \\
C(K_\epsilon)[\log k]^{2R} \xi^{-2d}k^{-\lambda}\leq \epsilon^{10} \int_{G^+}f_{+,\xi} dm_+\int_{G}\bar{\phi}^{+}_jd\mu, 
\end{multline*}
for $k$ large enough (largeness depending on $\epsilon$ and $\xi$). 
It remains to notice that by $V1.$, we have 
$$
\int_{G^+}f_{+,\xi} dm_+\int_{G}\bar{\phi}^{+}_jd\mu\leq (1+\epsilon^3)Leb^+(C(\xi,\mf g^+))\mu(D(\delta_k,x_j)).
$$
Combining the above inequalities, we get:
$$
Leb^+\left(H\in C(\xi,\mf g^+)\;:\; \psi^N\exp(H)x\in D\left(\delta_k,x_j\right)\right)\leq (1+\epsilon^3)^2(1+\epsilon^{10})
Leb^+(C(\xi,\mf g^+))\mu(D(\delta_k,x_j)).
$$
This give the LHS inequality in Lemma \ref{eq:distcube}. Analogously by considering $f_{-,\xi}$ and $\bar{\phi}^-_j$ we prove the right inequality. This finishes the proof of Lemma \ref{eq:distcube}.

\end{proof}

\end{document}